\documentclass[12pt]{article}
\usepackage[usenames,dvipsnames]{color}
\usepackage{xr,graphicx,amsthm,amsmath,amssymb,xargs,srcltx,bbm,todonotes,stmaryrd,framed,a4wide,accents,sectsty}
\usepackage[shortlabels]{enumitem}
\usepackage{framed}
\usepackage{pdfsync}

\usepackage[colorlinks]{hyperref}

\newcommand{\exampleend}{\hfill $\boxplus$}
\newcommand{\remarkend}{\hfill $\oplus$}

\newcommandx{\conditiondhsumpaper}[2]{\ensuremath{\mathcal{S}(#1,#2)}} 
\newcommandx\tedclustersl[2][1=n,2=\dhinterseq]{{\widetilde{\boldsymbol\mu}}^*_{#1,#2}} 

\newcommand\TEDclusterrandomsl{\widehat{\boldsymbol\mu}^*} 
\newcommand\TEDclusterrandomaltsl{\widehat{\widehat{\boldsymbol\mu}}^*}
\newcommand\TEPunivsl{\widetilde{\mathbb{T}}}

\newcommand\statinterseqn{k_n}

\subsubsectionfont{\itshape}

\usepackage{cleveref}

\newtheorem{theorem}{Theorem}[section]
\newtheorem{proposition}[theorem]{Proposition}
\newtheorem{lemma}[theorem]{Lemma}

\newtheorem{definition}[theorem]{Definition}

\theoremstyle{remark}
\newtheorem{example}[theorem]{Example}
\newtheorem{remark}[theorem]{Remark}

\crefname{hypothesis}{assumption}{assumptions}

\crefname{lemma}{Lemma}{Lemmas}
\crefname{theorem}{Theorem}{Theorems}
\crefname{proposition}{Proposition}{Propositions}
\crefname{exercise}{Exercise}{Exercises}
\crefname{corollary}{Corollary}{Corollaries}

\numberwithin{equation}{section}

\newcommandx\TEP[1][1=]{\mathbb{G}^{#1}} 
\newcommand\TEPuniv{\widetilde{\mathbb{T}}}

\newcommand\tepcluster{{\mathbb{G}}}
\newcommandx\tedcluster[2][1=n,2=\dhinterseq]{{\widetilde{\boldsymbol\nu}}^*_{#1,#2}} 
\newcommandx\tedclusterindep[2][1=n,2=\dhinterseq]{{\widetilde{\boldsymbol\nu}}^\indep_{#1,#2}} 
\newcommandx\tedclustermdep[2][1=n,2=\dhinterseq]{{\widetilde{\boldsymbol\nu}}^{*(m)}_{#1,#2}} 

\newcommand\TEDclusterrandom{\widehat{\boldsymbol\nu}^*} 

\newcommandx\tedclusterrandom[2][1=n,2=\dhinterseq]{{\widehat{\boldsymbol\nu}}^*_{#1,#2}}

\newcommand{\ANSJB}{{\rm ANSJB}}

\newcommand\constant{\mathrm{cst}}

\newcommandx\envelope[1][1={\mathbf H}]{{\mathbf {#1}}}

\newcommand\canditheta{{\vartheta}}
\newcommandx\stoploss[1][1={\rm stoploss}]{\theta_{#1}}
\newcommandx\stoplossest[1][1={{\rm stoploss},n}]{\widetilde\theta_{#1}}
\newcommandx\stoplossestk[1][1={{\rm stoploss},n,\statinterseq}]{\widehat\theta_{#1}}
\newcommandx\largedeviation[1][1={\rm largedev}]{\theta_{#1}}
\newcommandx\largedeviationest[1][1={{\rm largedev},n}]{\widetilde\theta_{#1}}
\newcommandx\largedeviationestk[1][1={{\rm largedev},n,\statinterseq}]{\widehat\theta_{#1}}
\newcommandx\ruin[1][1={\rm ruin}]{\theta_{#1}}
\newcommandx\ruinest[1][1={{\rm ruin},n}]{\widetilde\theta_{#1}}
\newcommandx\ruinestk[1][1={{\rm ruin},n,\statinterseq}]{\widehat\theta_{#1}}

\newcommandx\clfunc[1][1=h]{#1} 

\newcommandx\anticl[2][2=\dhinterseq]{\conditionS(#1,#2)}

\newcommandx\anticlpsi[3][1=\tepseq,2=\dhinterseq,3=\psi]{\conditionS(#1,#2,#3)}

\newcommandx{\norm}[2][1=]{\left|#2\right|_{#1}} 
\newcommandx{\lpnorm}[3][1=,3=]{\|#2\|_{#1}^{#3}}

\newcommandx{\matrixnorm}[2][2=]{\left\|#1\right\|_{#2}} 
\newcommandx{\matrixnormseries}[3][2=,3=]{\left\|#1\right\|_{#2,#3}} 
\newcommandx{\linftynorm}[2][1=]{\|#2\|_{#1}} 
\newcommandx{\anynorm}[2][2=]{\left|#1\right|_{#2}} 

\newcommandx{\supnormclass}[3][3=]{\left\|#1\right\|_{#2}^{#3}} 

\newcommandx{\sphere}[2][1=]{\mathbb{S}_{#1}^{#2}}
\newcommandx{\oball}[3][3=]{B_{#3}(#1,#2)} 
\newcommandx{\cball}[3][3=]{\overline{B}_{#3}(#1,#2)} 
\newcommandx{\coball}[3][3=]{B_{#3}^c(#1,#2)} 
\newcommandx{\ccball}[3][3=]{\overline{B}_{#3}^c(#1,#2)} 
\newcommandx\cone[1][1=C]{\mathcal{#1}}

\newcommandx\conej[1][1=\mathbf{j}]{\mathcal{C}_{#1}}   

\newcommandx{\coneindex}[2][2=\mathcal{C}]{#1_#2}

\newcommand\Nset{\mathbb{N}}
\newcommand\Zset{\mathbb{Z}}

\newcommand\Rset{\mathbb{R}}

\newcommandx\borel[1][1=\csms]{\mathcal{B}(#1)}   

\newcommand{\bszero}{{\boldsymbol0}}

\newcommand\indep{\dag}  

\newcommand\bsQ{\boldsymbol{Q}}

\newcommand\bsx{\boldsymbol{x}}
\newcommand\bsX{\boldsymbol{X}}

\newcommand\bsy{\boldsymbol{y}}
\newcommand\bsY{\boldsymbol{Y}}

\newcommand\bsTheta{\boldsymbol{\Theta}}

\newcommand\bsnu{\boldsymbol{\nu}}

\newcommand\tailmeasure{\bsnu} 
\newcommand\tailmeasurestar{{\bsnu}^*} 
\newcommand\ltwotmsmetric{\boldsymbol{\rho}^*} 

\newcommand\exc{\mce}  
\newcommand\anchor{\mca} 

\newcommandx{\numult}[2][2=]{\bsnu_{\boldsymbol{#1}_{#2}}} 

\newcommand\backest\Upsilon

\newcommandx{\nualphap}[2][1=\alpha,2=p]{\nu_{#1,#2}}  

\newcommandx\numultcondi[2][2=]{\boldsymbol{\nu}_{\boldsymbol{#1}_{#2}}}

\newcommandx\chain[1][1=Y]{\mathbb{#1}}

\newcommandx{\scalingseq}[1][1=n]{c_{#1}}  
\newcommandx{\scalingfunction}[1][1=]{c#1} 

 \newcommandx{\absquantileseq}[2][1=]{a_{#1#2}}

\newcommandx{\dhinterseq}[1][1=n]{r_{#1}}  
\newcommandx{\dhinterseqsmall}[1][1=n]{\ell_{#1}}  
\newcommandx{\dinterseq}[1][1=n]{r_{#1}}  
\newcommandx{\tepseq}[1][1=n]{u_{#1}}  

\newcommandx{\interseq}[1][1=n]{k_{#1}}  

\newcommandx{\scalingseqcone}[1][1=n]{c_{#1}} 

\newcommandx{\scalingseqhidden}[1][1=n]{\tilde{c}_{#1}}
\newcommandx{\scalingfunctionhidden}[1][1=]{\tilde{c}{#1}}
\newcommandx{\scalingseqcev}[1][1=n]{c^*_{#1}}

\newcommand\conditionS{\ensuremath{\mathcal{S}}} 
\newcommandx{\conditiondh}[2][1=\dhinterseq,2=\scalingseq]{\ensuremath{\mathcal{A}\mathcal{C}(#1,#2)}} 
\newcommandx{\conditionANSJB}[1][1=\scalingseq]{\mathrm{ANSJB}(\dhinterseq,#1)}
\newcommandx{\conditiondhsum}[1][1=\scalingseq]{\ensuremath{\mathcal{S}(\dhinterseq,#1)}} 
\newcommandx{\conditiondhsumW}[1][1=\scalingseq]{\ensuremath{\mathcal{SW}(\dhinterseq,#1)}} 

\newcommand\conditionR{\ensuremath{\mathcal{R}}} 

\newcommand\convdistr{\stackrel{\mbox{\tiny\rm d}}{\longrightarrow}} 

\newcommand\convprob{\stackrel{\tiny \mathbb{P}}{\longrightarrow}}

\newcommand\vaguelysharp{vaguely$^\#$}

\newcommandx\prohodistsym[1][1=]{\rho_{#1}}
\newcommandx\prohodist[3][3=]{\rho_{#3}(#1,#2)}

\newcommand\rme{\mathrm{e}}

\newcommand\rmd{\mathrm{d}} 

\newcommand\esp{\mathbb E}
\newcommand\pr{\mathbb P}
\newcommand\var{\mathrm{Var}}
\newcommand\cov{\mathrm{Cov}}

\newcommandx{\autocov}[1][1=]{\gamma_{#1}}

\newcommand{\tail}[1]{\overline{#1}}

\newcommandx{\cdfnorm}[1][1=\bsX]{H_{#1}}
\newcommandx{\tailcdfnorm}[1][1=\bsX]{\overline{H}_{#1}}

\newcommand\ind[1]{\mathbbm{1}{\left\{#1\right\}}}
\newcommand\indicsmall[1]{\mathbbm{1}_{\left\{#1\right\}}}

\newcommand\mca{\mathcal A}
\newcommand\mcb{\mathcal B}

\newcommand\mce{\mathcal E}
\newcommand\mcf{\mathcal F}
\newcommand\mcg{\mathcal G}

\newcommand\mch{\mathcal H}

\newcommand\mck{\mathcal K}
\newcommand\mcl{\mathcal L}

\newcommand\bbZ{\mathbb{Z}}

\newcommandx\test[2][1=X]{{#1}_{#2}}
\newcommandx\orderstat[3][1=X]{{#1}_{(#2:#3)}}
\newcommand\statinterseq{k}

\newcommand{\tcb}{\textcolor{blue}}

\newcommandx{\sequence}[3][2=\Zset,3=j]{\{#1_{#3},#3\in#2\}}
\newcommandx{\sequenceshort}[2][2=j]{\{#1_#2\}}
\newcommandx\sequ[3][2=j,3=\mathbb{Z}]{\{#1_#2,#2\in#3\}}
\newcommandx\sequnorm[3][3=j,2=\mathbb{Z}]{\{\norm{#1_#3},#3\in#2\}}
\newcommandx\sequnormq[4][2=,4=j,3=\mathbb{Z}]{\{\norm{#1_#4}^{#2},#4\in#3\}}

\newcommandx\uncompactd[2][1=d]{(\overline{\Rset}^{#1})^{#2}\setminus\{\boldsymbol0\}}
\newcommandx{\barrsetproduct}[2][1=d]{(\overline{\Rset}^{#1})^{#2}}
\newcommandx{\rsetproduct}[2][1=d]{(\Rset^{#1})^{#2}}

\newcommand{\metricspace}{\csms}

\newcommand\spaceD{\mathbb{D}}  

\newcommandx\csms[1][1=E]{\mathsf{#1}}   
\newcommandx\borelcsms[1][1=E]{\mathcal{#1}}   
\newcommandx\mplusx[1][1=]{\mathcal{M}#1}  
\newcommandx\mplusxp[1][1=]{\mathcal{N}{#1}} 
\newcommandx\mplusxpb[1][1=\borelcsms]{\mathcal{N}_{pb}({#1})}  
\newcommandx\mplusxpone[1][1=\borelcsms]{\mathcal{N}_{p1}({#1})}  
\newcommandx\mplusxpeps[1][1=\borelcsms]{\mathcal{N}_{p\epsilon}({#1})}  
\newcommandx\mplusxf[1][1=\borelcsms]{\mathcal{M}_f({#1})}
\newcommandx\mplusxpf[1][1=\borelcsms]{\mathcal{N}_{pf}({#1})} 
\newcommandx\mplusxps[1][1=\borelcsms]{\mathcal{N}_{ps}({#1})} 
\newcommandx\mplusxpS[1][1=\borelcsms]{\mathcal{N}_{pS}({#1})} 
\newcommandx\mplusxpsc[1][1=\borelcsms]{\mathcal{N}_{psc}({#1})} 

\newcommand\distance{\mathrm{d}}  
\newcommandx\metric[1][1=\metricspace]{\distance_{#1}} 
\newcommandx\metricmcg{\rho} 

\newcommandx\bracknum[3][2=\mch]{N_{[\,]}(#1,#2,#3)} 
\newcommandx\bracknumarray[2][2=\mch]{N_{[\,]}(#1,#2,L^2_n)} 
\newcommandx\entropynum[3][3=\mch]{N(#1,#3,#2)} 

\newcommandx\process[1][1=X]{\mathbb{#1}}

\newcommandx\hillest[3][1=n,3=]{\widehat{\gamma}_{#1,#2}^{#3}}
\newcommandx\hillmoment[2][1=n]{\widehat{\gamma}_{#1,#2}^{(M)}}

\newcommand\TEPclusterlimit{\mathbb{G}} 
\newcommand\lzero{\ell_0}
\newcommand\tildelzero{\tilde{\ell}_0}

\newcommand\lone{\ell_1}

\newcommandx\lalpha[1][1=\alpha]{\ell_{#1}}

\newcommand\shift{B}

\newcommand\iid{i.i.d.}
\newcommand\withoutlog{without loss of generality}

\newcommand\wrt{with respect to}
\newcommand\ie{i.e.}

\newcommand\semimetric{semi-metric}
\newcommand\rhs{right-hand side}

\newcommand\fidi{finite-dimensional}

\newcommand\realvalued{real-valued}

\newcommand\nonnegative{non-negative}

\newcommand\shiftinvariant{shift-invariant}

\newcommand\Qseq{conditional spectral tail process}

\begin{document}
\title{Estimation of cluster functionals for regularly varying time series: sliding blocks estimators}

\author{Youssouph Cissokho\thanks{University of Ottawa}\and Rafa{\l} Kulik\thanks{University of Ottawa} }

\date{\today}
\maketitle

\begin{abstract}
Cluster indices describe extremal behaviour of stationary time series.
We consider their sliding blocks estimators. Using a modern theory of multivariate, regularly varying time series, we obtain central limit theorems under conditions that can be easily verified for a large class of models. In particular, we show that in the Peak over Threshold framework, sliding and disjoint blocks estimators have the same limiting variance.
\end{abstract}

\section{Introduction}
Consider
a stationary, regularly varying $\Rset^d$-valued time series $\bsX=\sequence{\bsX}$. We are interested in estimating cluster indices that describe its extremal behaviour. Informally speaking, a cluster is a triangular array $(\bsX_1/\tepseq,\ldots,\bsX_{\dhinterseq}/\tepseq)$ with $\dhinterseq,\tepseq\to\infty$ that converges in distribution in a certain sense. Cluster indices are obtained by applying the appropriate  functional $H$ to the cluster.
The functionals are defined on $(\Rset^d)^\Zset$ and are such that their values do not depend on coordinates that are equal to zero.
More precisely, for $\bsX=\{\bsX_j,j\in\Zset\}\in (\Rset^d)^\Zset$ and $i\leq j\in\Zset$,  we denote $\bsX_{i,j}=(\bsX_i,\ldots, \bsX_j)\in (\Rset^d)^{(j-i+1)}$.
Then, we identify $H(\bsX_{i,j})$ with
$H((\bszero,\bsX_{i,j},\bszero))$, where $\bszero\in (\Rset^d)^\Zset$ is the zero sequence. Such functionals $H$ will be called \textit{cluster functionals}.

Let $\norm{\cdot}$ be an arbitrary norm on $\Rset^d$ and $\{\tepseq\}$, $\{\dhinterseq\}$ be such that
\begin{align}\label{eq:rnbarFun0}
\lim_{n\to\infty}\tepseq=\lim_{n\to\infty}\dhinterseq =\lim_{n\to\infty}n\pr(\norm{\bsX_0}>\tepseq) = \infty\;, \ \ \lim_{n\to\infty}\dhinterseq/n=\lim_{n\to\infty}\dhinterseq\pr(\norm{\bsX_0}>\tepseq) = 0\;.
\tag{$\conditionR(\dhinterseq,\tepseq)$}
\end{align}
Given a cluster functional $H$ on $(\Rset^d)^\Zset$, we want to estimate the limiting
quantity
\begin{align}
\tailmeasurestar(H)=  \lim_{n\to\infty}  \tailmeasurestar_{n,\dhinterseq}(H)=\lim_{n\to\infty} \frac{\esp[H(\bsX_{1,\dhinterseq}/\tepseq)]}{\dhinterseq \pr(\norm{\bsX_0}>\tepseq)}\;.  \label{eq:thelimitwhichisnolongercalledbH}
\end{align}
To guarantee existence of the limit we will require additional anticlustering assumptions on the time series $\sequence{\bsX}$. The cluster indices of interest are, among others:
\begin{itemize}
\item the extremal index obtained with $H_1(\bsx)=\ind{\bsx^*>1}$, $\bsx=\{\bsx_j,j\in\Zset\}\in(\Rset^d)^\Zset$;
\item the cluster size distribution obtained with
\begin{align}\label{eq:h2}
  H_2(\bsx)=\ind{\sum_{j\in\Zset}\ind{\norm{\bsx_j}>1}=m}\;, \ \   \bsx=\{\bsx_j,j\in\Zset\}\in(\Rset^d)^\Zset\;, \ \  m\in\Nset \;;
  \end{align}
\item the stop-loss index of a univariate time series obtained with
\begin{align}\label{eq:h3}
H_3(\bsx) = \ind{\sum_{j\in\Zset} (x_j-1)_+ > \eta}\;, \ \    \bsx=\{\bsx_j,j\in\Zset\}\in\Rset^\Zset\;, \ \    \eta>0\;;
\end{align}
\item the large deviation index of a univariate time series obtained with
 \begin{align}\label{eq:h4}
 H_4(\bsx)=\ind{K(\bsx)>1}\;, \ \  K(\bsx) = \left(\sum_{j\in\Zset} x_j\right)_+\;, \ \
 \bsx=\{\bsx_j,j\in\Zset\}\in\Rset^\Zset\;;
  \end{align}
\item the ruin index of a univariate time series obtained with
\begin{align}\label{eq:h5}
  H_5(\bsx)=\ind{K(\bsx)>1}\;, \ \   K(\bsx) = \sup_{i\in\Zset} \left(\sum_{j\leq i}
    x_j\right)_+\;, \ \ \bsx=\{\bsx_j,j\in\Zset\}\in\Rset^\Zset\;.
  \end{align}
\end{itemize}
We note that the extremal index is the classical quantity that arises in the extreme value theory for dependent sequences, the large deviation index was studied under the name \textit{cluster index} in \cite{mikosch:wintenberger:2013,mikosch:wintenberger:2014}, the cluster size distribution is again a well-known object and was studied in \cite{hsing:1991} and \cite{drees:rootzen:2010}, while the remaining cluster indices seem to be new.

Several methods of estimation of the limit $\tailmeasurestar(H)$ in \eqref{eq:thelimitwhichisnolongercalledbH} may be
employed. The natural one is to consider a
statistics based on disjoint blocks of size $\dhinterseq$, cf. \cite{drees:rootzen:2010} and \cite{kulik:soulier:2020},
\begin{align*}
 \tedcluster(H):= \frac{1}{n\pr(\norm{\bsX_0}>\tepseq)}   \sum_{i=1}^{m_n}
  H(\bsX_{(i-1)\dhinterseq+1,i\dhinterseq}/\tepseq) \;,
\end{align*}
where $m_n=[n/\dhinterseq]$. The data-based estimator is constructed as follows. Let $\statinterseqn\to\infty$ be a sequence of integers and define $\tepseq$ by   $\statinterseqn=n\pr(\norm{\bsX_0}>\tepseq)$. Let $\orderstat[\norm{\bsX}]{n}{1}\leq \cdots \leq \orderstat[\norm{\bsX}]{n}{n}$ be order statistics from $\norm{\bsX_1},\ldots,\norm{\bsX_n}$.
Define
\begin{align}\label{eq:feasible-estimator-of-cluster-index-consistency}
  \TEDclusterrandom_{n,\dhinterseq}(H)
  &:= \frac{1}{\statinterseqn} \sum_{i=1}^{m_n} H(\bsX_{(i-1)\dhinterseq+1,i\dhinterseq}/\orderstat[\norm{\bsX}]{n}{n-\statinterseqn}) \; .
\end{align}
Although some special cases were considered (estimation of the extremal index in \cite{hsing:1991} and \cite{smith:weissman:1994}; tail array sums in \cite{rootzen:leadbetter:dehaan:1998}), the general theory was developed in \cite{drees:rootzen:2010}. The summary of the theory for the disjoint blocks estimators can be found in \cite[Chapter 10]{kulik:soulier:2020}, where consistency and the central limit theorems are established.

In this paper we consider the sliding blocks statistics
\begin{align}\label{eq:sliding-block-estimator-nonfeasible-1}
\tedclustersl(H):=\frac{1}{q_n \dhinterseq \pr(\norm{\bsX_0}>\tepseq)}\sum_{i=0}^{q_n-1}H\left(\bsX_{i+1,i+\dhinterseq}/\tepseq\right)\;,
\end{align}
where $q_n=n-\dhinterseq-1$
and the corresponding estimator defined in terms of order statistics:
\begin{align}\label{eq:sliding-block-estimator-feasible}
\TEDclusterrandomsl_{n,\dhinterseq}(H)=\frac{1}{\dhinterseq\statinterseqn}
\sum_{i=0}^{q_n-1}H\left(\bsX_{i+1,i+\dhinterseq}/\orderstat[\norm{\bsX}]{n}{n-\statinterseqn}\right)\;.
\end{align}

The sliding blocks estimators have been studied for some specific functionals $H$, however there has been no unified theory available. Recently, \cite{drees:neblung:2020} used the framework of \cite{drees:rootzen:2010} and showed that the limiting variance of the sliding blocks estimator never exceeds that of the disjoint blocks estimator. In case of the extremal index, both variances are equal.

The goal of this paper is to obtain the asymptotic normality of the sliding blocks estimators.
Our focus is on providing the conditions that can be easily verified for a variety of time series models. At the same time, we will show that the limiting variance of both disjoint and sliding blocks estimators is the same. To achieve our goal, we combine \cite{drees:rootzen:2010} approach with the modern theory of stationary, regularly varying time series.

In order to proceed, in \Cref{sec:prel} we fix the notation, recall the notion of the tail process associated to a stationary regularly varying time series; and introduce the cluster indices.

Next, we need to answer a non-trivial question: \textit{When does the limit $\tailmeasurestar(H)$ exist?}.
For this,
\Cref{sec:weakconv-cluster-measure} deals with convergence of cluster measures and cluster indices $\tailmeasurestar(H)$ appear as the limit. Existence of the limit requires an anticlustering assumption. In conjunction with a particular choice of functionals, we will be in position to give specific examples of cluster indices. The contents of this section is based on
\cite[Chapter 6]{kulik:soulier:2020}.
Some results stem from \cite{mikosch:wintenberger:2014,mikosch:wintenberger:2016} and  \cite{basrak:planinic:soulier:2018}.

The main result is \Cref{thm:sliding-block-clt-1}. We prove the central limit theorem for the data-based sliding blocks estimator \eqref{eq:sliding-block-estimator-feasible} under easy to verify assumptions.
Those conditions can be verified for a variety of models: regularly varying functions of Markov chains, infinite order moving averages, max-stable processes. See \cite{kulik:soulier:wintenberger:2018} and \cite[Part III]{kulik:soulier:2020}.

The most important (and somehow surprising) conclusion is that both sliding \eqref{eq:sliding-block-estimator-feasible} and disjoint \eqref{eq:feasible-estimator-of-cluster-index-consistency} blocks estimators yield the same variance. This is in agreement with the result for the extremal index in \cite{drees:neblung:2020}. On the other hand, it seems to be a contradiction with other available results. We explain this in \Cref{sec:comments}.

All proofs are included in \Cref{sec:proofs}.

\section{Preliminaries}\label{sec:prel}
In this section we fix the notation and introduce the relevant classes of functions. In \Cref{sec:tail-process} we recall the notion of the tail and the spectral tail process (cf. \cite{basrak:segers:2009}). In \Cref{sec:cluster-index} we define cluster indices; see \cite[Chapter 5]{kulik:soulier:2020} for a detailed introduction.
\subsection{Notation}
Let $|\cdot|$ be a norm on $\Rset^d$.
For a sequence $\bsx=\{\bsx_j,j\in\Zset\}\in (\Rset^d)^\Zset$ and $i\leq j\in\Zset\cup\{-\infty,\infty\}$ we denote $\bsx_{i,j}=(\bsx_i,\ldots,\bsx_j)\in (\Rset^d)^{j-i+1}$, $\bsx_{i,j}^*=\max_{i\leq l\leq j}|\bsx_l|$ and $\bsx^*=\sup_{j\in\Zset}|\bsx_j|$. By $\bszero$ we denote the zero sequence; its dimension can be different of each of its occurrences.

By $\lzero(\Rset^d)$ we denote the set of $\Rset^d$-valued sequences which tend to zero at infinity. Likewise, $\lone(\Rset^d)$
consists of sequences such that $\sum_{j\in\Zset}|\bsx_j|<\infty$.

We will use the blocking method. If $\bsX$ is a time series of interest, then  $(\bsX_1^\indep,\dots,\bsX_n^\indep)$ is a pseudo-sample
  such that the blocks $(\bsX_{(i-1)\dhinterseq+1}^\indep,\dots,\bsX_{i\dhinterseq}^\indep)$,
  $i=1,\ldots,m_n=[n/\dhinterseq]$, are mutually independent with the same distribution as the original block
  $(\bsX_1,\dots,\bsX_{\dhinterseq})$.

\subsection{Classes of functions}\label{sec:classes}
Functionals $H$ are defined on $\lzero(\Rset^d)$ with the convention $H(\bsx_{i,j})=H((\bszero,\bsx_{i,j},\bszero))$. In particular,
the map $\exc$ is defined on $\lzero(\Rset^d)$ by
$\exc(\bsx) = \sum_{j\in\Zset} \ind{\norm{\bsx_j}>1}$. For $s>0$, the function $H_s:(\Rset^d)^\Zset\to\Rset$ is defined by $H_s(\bsx)=H(\bsx/s)$.
We consider the following classes:
\begin{itemize}
\item $\mcl$ is the class of bounded \realvalued\ functions defined on $(\Rset^d)^\Zset$ that are either
Lipschitz continuous \wrt\ the uniform norm or almost
surely continuous with respect to the distribution of the tail process $\bsY$.
This class includes functions like $\ind{\bsx^*>1}$, $\ind{\sum_{j\in\Zset}|\bsx_j|>1}$.
See Remark 6.1.6 in \cite{kulik:soulier:2020}.

\item $\mca\subset\mcl$ is the class of \shiftinvariant\ functionals with support separated
  from $\bszero$. In particular, for $H\in \mca$, $H(\bszero)=0$. The class $\mca$ includes $\ind{\bsx^*>1}$.
\item $\mck$ is the class of shift-invariant functionals $K:(\Rset^d)^\Zset\to\Rset$ defined on
$\lone(\Rset^d)$ such that $K(\bszero)=0$ and which are
  Lipschitz continuous with constant $L_K$, \ie\,
  \begin{align}
    |K(\bsx) - K(\bsy)| \leq L_K \sum_{j\in\Zset} \norm{\bsx_j-\bsy_j}\;, \ \ \bsx,\bsy\in\lone(\Rset^d) \; .  \label{eq:hypo-K}
  \end{align}
\item $\mcb\subset\mcl$ is the class of functionals $H$ of the form $H=\ind{K>1}$,
where $K\in\mck$.
Functionals in $\mcb$ may have support which is not
separated from $\bszero$. The typical example is $H(\bsx)=\ind{\sum_{j}|\bsx_j|>1}$; note that $H\not\in\mca$.
\end{itemize}

\subsection{Tail and spectral tail process}\label{sec:tail-process}
Let $\bsX=\sequence{\bsX}$ be a stationary, regularly varying time series with values in $\Rset^d$ and tail index $\alpha$. In particular,
\begin{align*}
\lim_{x\to\infty}\frac{\pr(|\bsX_0|> tx)}{\pr(|\bsX_0|> x)}=t^{-\alpha}
\end{align*}
for all $t>0$.
 Then, there exists a sequence $\bsY=\sequence{\bsY}$ such that
\begin{align*}
  \pr(x^{-1}(\bsX_i,\dots,\bsX_j) \in \cdot \mid |\bsX_0|>x)   \mbox{ converges weakly to }  \pr((\bsY_i,\dots,\bsY_j) \in \cdot)
\end{align*}
as $x\to\infty$ for all $i\leq j\in\Zset$.
We call $\bsY$ the tail process.
See \cite{basrak:segers:2009}. Equivalently, viewing $\bsX$ and $\bsY$ as random elements with values in $(\Rset^d)^\Zset$, we have for every bounded or
\nonnegative\ functional~$H$ on $(\Rset^d)^{\Zset}$, continuous \wrt\ the product topology,
\begin{align*}
   \lim_{x\to\infty} \frac{\esp[H(x^{-1}\bsX)\ind{\norm{\bsX_0}>x}]} {\pr(\norm{\bsX_0}>x)}
  & =  \esp[H(\bsY)] \; .
\end{align*}
%
Define
$\bsTheta_j = {\bsY_j}/{|\bsY_0|}$, $j\in\Zset$.
The sequence $\bsTheta=\sequence{\bsTheta}$ is called the spectral tail
process. The random variable $|\bsY_0|$ has the Pareto distribution with index $\alpha$ and is independent from $\bsTheta$. Hence
for a \nonnegative\ measurable function $H:(\Rset^d)^\Zset\to\Rset$,
\begin{align}
  \esp[H(\bsY)]
  & = \int_1^\infty  \esp[H(r\bsTheta)] \alpha r^{-\alpha-1} \rmd r \; . \label{eq:polar2}
\end{align}
\subsection{Cluster measure and cluster indices}\label{sec:cluster-index}
Consider
the infargmax functional $\anchor_0$ defined on $(\Rset^d)^\Zset$ by
    $\anchor_0(\bsy)=\inf\{j:\bsy_{-\infty,{j}}^*=\bsy^*\}$, with the convention that
    $\inf\emptyset=+\infty$.
If  $\pr(\anchor_0(\bsY)\notin\Zset)=0$ then we can define
\begin{align}
  \label{eq:canditheta-anchor}
  \canditheta = \pr(\anchor_0(\bsY)=0) \; .
\end{align}
In fact, $\anchor_0$ can be replaced with any anchoring map (see \cite{planinic:soulier:2018} and \cite[Theorem 5.4.2]{kulik:soulier:2020}). In particular,
\begin{align*}
\canditheta=\pr(\anchor_0(\bsY)=0)=\pr\left(\sup_{j\leq -1}|\bsY_j|\leq 1\right)=\pr\left(\sup_{j\geq 1}|\bsY_j|\leq 1\right)\;.
\end{align*}
Therefore, $\canditheta$ can be recognized as the (candidate) extremal index. It becomes the usual extremal index under additional mixing and anticlustering conditions.

\begin{definition}[Cluster measure]
  \label{def:clustermeasure}
  Let $\bsY$ and $\bsTheta$ be the tail process and the spectral tail process, respectively, such that $\pr(\lim_{|j|\to\infty} \bsY_j=\bszero)=1$. The {cluster measure} is the measure $\tailmeasurestar$ on
  $\lzero(\Rset^d)$ defined by
  \begin{align}
    \label{eq:def-tailmeasurestar-premiere}
\tailmeasurestar    = \canditheta \int_0^\infty \esp[\delta_{r\bsTheta}\ind{\anchor_0(\bsTheta)=0}] \alpha r^{-\alpha-1} \rmd r   \; .
\end{align}
\end{definition}
The measure $\tailmeasurestar$ is
boundedly finite on $(\Rset^d)^\Zset\setminus\{\bszero\}$, puts no mass at $\bszero$ and is
$\alpha$-homogeneous.
%
%
Furthermore, the cluster measure can be expressed in terms of another sequence.
\begin{definition}
    \label{def:sequence-Q}
    Assume that $\pr(\anchor_0(\bsY)\notin\Zset)=0$. The \Qseq\ $\bsQ$ is a random sequence with the distribution of
    $(\bsY^*)^{-1}\bsY$ conditionally on $\anchor_0(\bsY)=0$.
  \end{definition}
The sequence $\bsQ$ appeared implicitly in the seminal
paper \cite{davis:hsing:1995}.
See also \cite{basrak:segers:2009}, \cite[Definition 3.5]{planinic:soulier:2018} and \cite[Chapter 5]{kulik:soulier:2020}. An abstract setting is considered in \cite{dombry:hashorva:soulier:2018}.

Note that $\anchor_0(\bsY)=0$ if and only if $\anchor_0(\bsTheta)=0$. Then also $\bsY^*=\norm{\bsY_0}$. Thus, \eqref{eq:def-tailmeasurestar-premiere} and the definition of $\bsQ$ give for a bounded or \nonnegative\ measurable function $H$ on $\lzero(\Rset^d)$,
\begin{align}\label{eq:seq-Q}
 \tailmeasurestar(H) & =
 \vartheta \int_0^\infty \esp[H(r\bsQ)]  \alpha r^{-\alpha-1} \rmd r =
 \vartheta \int_0^\infty \esp[H(r\bsTheta)\ind{\anchor_0(\bsTheta)=0}]  \alpha r^{-\alpha-1} \rmd r
 \; .
\end{align}
If moreover $H$ is such that $H(\bsy)=0$ if
$\bsy^*\leq\epsilon$ for one $\epsilon>0$, then
\begin{align}
  \label{eq:relation-cluster-Y-Q-Theta-epsilon}
  \tailmeasurestar(H) = \epsilon^{-\alpha} \esp[H(\epsilon\bsY) \ind{\anchor_0(\bsY)=0}] \; .
\end{align}
Comparing \eqref{eq:def-tailmeasurestar-premiere} or \eqref{eq:relation-cluster-Y-Q-Theta-epsilon} with
\eqref{eq:polar2} we can see that the $\tailmeasurestar(H)$ does not agree with $\esp[H(\bsY)]$. The additional indicator comes essentially from the conditioning on the location of the maximum of the sequence $\bsY$.

\begin{definition}[Cluster index]
\label{def:cluster-index}
We will call $\tailmeasurestar(H)$ the cluster index associated to the
functional $H$.
\end{definition}

\section{Convergence of cluster measure}\label{sec:weakconv-cluster-measure}
Recall \ref{eq:rnbarFun0}.
Define the measures $\tailmeasurestar_{n,\dhinterseq}$, $n\geq1$, on $\lzero(\Rset^d)$ as follows:
\begin{align*}
  \tailmeasurestar_{n,\dhinterseq}
  & =  \frac{1} {\dhinterseq \pr(\norm{\bsX_0}>\tepseq)} \esp \left[\delta_{\tepseq^{-1}\bsX_{1,\dhinterseq}} \right] \; .
\end{align*}
We are interested in convergence of $\tailmeasurestar_{n,\dhinterseq}$ to
$\tailmeasurestar$. The results of this section are extracted from \cite[Chapter 6]{kulik:soulier:2020}.
See also \cite{planinic:soulier:2018} and \cite{basrak:planinic:soulier:2018}.
\subsection{Anticlustering condition}
For each fixed $r\in\Nset$, the distribution of
$\tepseq^{-1}\bsX_{-r,r}$ conditionally on $\norm{\bsX_0}>\tepseq$ converges weakly to the
distribution of $\bsY_{-r,r}$. In order to let $r$ tend to infinity, we must embed all these finite
vectors into one space of sequences. By adding zeroes on each side of the vectors
$\tepseq^{-1}\bsX_{-r,r}$ and $\bsY_{-r,r}$ we identify them with elements of the space
$\lzero(\Rset^d)$. Then $\bsY_{-r,r}$
converges (as $r\to\infty$) to $\bsY$ in $\lzero(\Rset^d)$ if (and only if) $\bsY\in\lzero(\Rset^d)$ almost surely.
%
However, this is not enough for statistical purposes and we consider the following definition.
\begin{definition}
  [\cite{davis:hsing:1995}, Condition~2.8]\label{def:DH}
  Condition~\ref{eq:conditiondh}
  holds if for all $x,y>0$,
  \begin{align}
    \label{eq:conditiondh}
    \lim_{k\to\infty} \limsup_{n\to\infty}\pr\left(\max_{k\le |j|\leq  \dhinterseq}|\bsX_j|
    > \tepseq x\mid |\bsX_0|> \tepseq y \right)=0 \; .
    \tag{$\conditiondh[\dhinterseq][\tepseq]$}
  \end{align}
\end{definition}
Condition \ref{eq:conditiondh} is referred to as the anticlustering condition. It is fulfilled by many models, including geometrically ergodic Markov chains,
short-memory linear or max-stable processes.
\ref{eq:conditiondh} implies that $\bsY\in \lzero(\Rset^d)$. Its
main consequence is the following result.
  \begin{proposition}[\cite{basrak:segers:2009}, Proposition~4.2; \cite{kulik:soulier:2020}, Theorem 6.1.4]
  \label{lem:tailprocesstozero}
  Let $H\in \mcl$.
  If Condition~\ref{eq:conditiondh}
  holds, then
  \begin{align*}
    \lim_{n\to\infty} \esp[H(\tepseq^{-1} \bsX_{-\dhinterseq,\dhinterseq})\mid \norm{\bsX_0}>\tepseq] = \esp[H(\bsY)] \; .
  \end{align*}
\end{proposition}
Condition \ref{eq:conditiondh} holds for sequence of \iid\ random
variables whenever $\lim_{n\to\infty} \dhinterseq\pr(\norm{\bsX_0}>\tepseq)=0$, which can be recognized as on the restrictions imposed in \ref{eq:rnbarFun0}
\subsection{Vague convergence of cluster measure}
We now investigate the unconditional convergence of
$\tepseq^{-1}\bsX_{1,\dhinterseq}$. Contrary to
\Cref{lem:tailprocesstozero}, where an extreme value was imposed at time 0, a large value in the
cluster can happen at any time. Moreover, the convergence of $\tailmeasurestar_{n,\dhinterseq}(H)$ to
$\tailmeasure^*(H)$ may hold only for \shiftinvariant\ functionals $H$. Therefore, we need the following definition.
\begin{definition}
  \label{def:spacetildelzero}
  The space $\tildelzero(\Rset^d)$ is the space of equivalence classes of $\lzero(\Rset^d)$ endowed with the
  equivalence relation $\sim$ defined by
  \begin{align*}
    \bsx\sim \bsy \Longleftrightarrow \exists j \in\Zset \; , \ \shift^j\bsx=\bsy \; ,
  \end{align*}
  where $\shift$ is the backshift operator.
\end{definition}
The proof of the next result is given in \Cref{sec:proofs}.
\begin{proposition}
\label{theo:cluster-RV}
  Let  condition~\ref{eq:conditiondh} hold. The sequence of measures $\tailmeasurestar_{n,\dhinterseq}$, $n\geq1$ converges
    \vaguelysharp\ on $\tildelzero(\Rset^d)\setminus\{\bszero\}$ to~$\tailmeasurestar$, that is,
    for all $H\in \mca$,
  \begin{align*}
    \lim_{n\to\infty} \tailmeasurestar_{n,\dhinterseq}(H)
    = \lim_{n\to\infty}\frac{\esp[ H(\tepseq^{-1}\bsX_{1,\dhinterseq})]} {\dhinterseq \pr(\norm{\bsX_0}>\tepseq)}
    = \tailmeasure^*(H) \; .
  \end{align*}
\end{proposition}
The immediate consequence is the following limit (cf. \eqref{eq:canditheta-anchor}):
  \begin{align*}
    \lim_{n\to\infty} \frac{\pr(\bsX_{1,\dhinterseq}^*>\tepseq)}{\dhinterseq\pr(\norm{\bsX_0}>\tepseq)} = \canditheta  \; .
  \end{align*}
Since $H_2,H_3\in\mca$ (cf. \eqref{eq:h2}-\eqref{eq:h3}), we can introduce the following cluster indices.
\begin{example}[Cluster size distribution]
  \label{xmpl:clustersizedistribution}
  If~\ref{eq:conditiondh}
  holds, \Cref{theo:cluster-RV} yields
  \begin{multline*}
    \lim_{n\to\infty} \pr\left( \sum_{j=1}^{\dhinterseq}
    \ind{\norm{\bsX_j}>\tepseq}=m \mid \bsX_{1,\dhinterseq}^* > \tepseq \vphantom{\sum_{i=1}^{\dhinterseq}} \right) \\
     = \pr\left(\sum_{j\in\Zset} \ind{\norm{\bsY_j}>1} =m\mid \bsY_{-\infty,-1}^* \leq 1 \right)=:\pi(m) \; .
   \end{multline*}
  \exampleend
\end{example}

\begin{example}[Stop-loss index]
  \label{xmpl:stop-loss}
  Consider a univariate time series.
   Define the {stop-loss index}:
  \begin{align*}
    \stoploss(\eta)=& \lim_{n\to\infty} \frac{\pr\left( \sum_{j=1}^{\dhinterseq}(X_j-\tepseq)_+ >
                \eta\tepseq \right)} {\dhinterseq\pr(X_0>\tepseq)}  = \pr\left( \sum_{j=0}^\infty(Y_j-1)_+ > \eta ,\bsY_{-\infty,-1}^*\leq 1\right) \;.
  \end{align*}
  This index seems to be new.
  \exampleend
\end{example}

\subsection{Indicator functionals not vanishing around zero}\label{sec:extension-1}
\Cref{theo:cluster-RV} entails convergence of $\tailmeasurestar_{n,\dhinterseq}(H)$ for
$H\in\mca$.
For functionals which are not
defined on the whole space $\lzero(\Rset^d)$, such as $H_4$ and $H_5$ from \eqref{eq:h4}-\eqref{eq:h5}, we need an additional assumption on Asymptotic Negligibility of Small Jumps.
\begin{definition}
  \label{def:AN-small}
  Condition~\ref{eq:AN-small} holds if for all
  $\eta>0$, \index{$\ANSJB(\dhinterseq,\tepseq)$}
  \begin{align}
    \lim_{\epsilon\to0} \limsup_{n\to\infty}
    \frac{\pr(\sum_{j=1}^{\dhinterseq} |\bsX_j| \ind{|\bsX_j|\leq\epsilon \tepseq}>\eta \tepseq)}
    {\dhinterseq\pr(|\bsX_0|>\tepseq)} = 0 \; . \tag{$\ANSJB(\dhinterseq,\tepseq)$} \label{eq:AN-small}
  \end{align}
\end{definition}
The proofs of the next two results are given in \Cref{sec:proofs}.
\begin{lemma}
  \label{lem:summability-Q}
  If \ref{eq:conditiondh} and \ref{eq:AN-small} hold, then
  \begin{align*}
    \lim_{n\to\infty} \frac{\pr(\sum_{i=1}^{\dhinterseq} \norm{\bsX_j}>\tepseq )} {\dhinterseq\pr(\norm{\bsX_0}>\tepseq )}
    =   \esp\left[ \left(\sum_{j\in\Zset} |\bsQ_j|\right)^\alpha \right]<\infty  \; .
  \end{align*}
\end{lemma}

\begin{proposition}
  \label{theo:limit-lone-Q}
  Assume that \ref{eq:conditiondh} and \ref{eq:AN-small} hold.   Then for $K\in{\mathcal K}$,
  \begin{align*}
   \tailmeasurestar(\ind{K>1})= \lim_{n\to\infty} \frac{\pr(K(\bsX_{1,\dhinterseq}/\tepseq )>1)} {\dhinterseq\pr(\norm{\bsX_0}>\tepseq )} = \canditheta
    \int_0^\infty \pr(K(z\bsQ)>1) \alpha z^{-\alpha-1} \rmd z  < \infty \; .  
  \end{align*}
\end{proposition}
If $K$ is a 1-homogeneous
satisfying the assumptions of \Cref{theo:limit-lone-Q}, then
\begin{align*}
  \tailmeasurestar(\ind{K>1})
  &   = \canditheta\esp[K_+^\alpha(\bsQ)]
    = \esp[K_+^\alpha(\bsTheta_{0,\infty})-K_+^\alpha(\bsTheta_{1,\infty})] \; .  
\end{align*}
\begin{example}[Large deviations index]
\label{xmpl:large-deviation}
Let $\sequence{X}$ be an univariate time series. The functional $H_4$ defined in \eqref{eq:h4}
  yields the large deviations index:
  \begin{align*}
    \largedeviation
    & = \lim_{n\to\infty} \frac{\pr\left(\left(\sum_{j=1}^{\dhinterseq} X_j\right)_+>\tepseq \right)}{\dhinterseq\pr(\norm{X_0}>\tepseq )}
      = \esp \left[ \left(\sum_{j=0}^\infty \Theta_j\right)_+^{\alpha} -       \left(\sum_{j=1}^\infty \Theta_j\right)_+^{\alpha} \right] \; .
  \end{align*}
  The index $\largedeviation$, under the name \textit{cluster index}, was introduced in \cite{mikosch:wintenberger:2016}.
  \exampleend
\end{example}

\begin{example}[Ruin index]
  \label{xmpl:ruine-cluster-index}
  Take $H_5$ defined in \eqref{eq:h5}.
  \Cref{theo:limit-lone-Q} gives
  \begin{align*}
    \ruin=\lim_{n\to\infty} \frac{\pr(\max_{1\leq j\leq \dhinterseq} \sum_{i=1}^j X_i>\tepseq )}{\dhinterseq\pr(\norm{X_0}>\tepseq )}
    = \canditheta  \esp \left[ \sup_{i\in\Zset} \left(\sum_{j\leq i} Q_j\right)_+^\alpha \right] \; .   
  \end{align*}
  \exampleend
\end{example}


\section{Central limit theorem for blocks estimators}\label{sec:main-results}
\subsection{Sliding blocks estimators}
Let $q_n=n-\dhinterseq+1$. Thanks to \Cref{theo:cluster-RV} and  \Cref{theo:limit-lone-Q}, we have for $H\in\mca\cup\mcb$,
\begin{align*}
&\lim_{n\to\infty}\esp\left[\frac{\sum_{i=0}^{q_n-1}H\left(\bsX_{i+1,i+\dhinterseq}/\tepseq\right)}{q_n \dhinterseq \pr(\norm{\bsX_0}>\tepseq)}\right]=\lim_{n\to\infty}\frac{\esp\left[H\left(\bsX_{1,\dhinterseq}/\tepseq\right)\right]}{\dhinterseq \pr(\norm{\bsX_0}>\tepseq)}=
\lim_{n\to\infty}\tailmeasurestar_{n,\dhinterseq}(H)=\tailmeasurestar(H)\;.
\end{align*}
This indicates that a consistent pseudo-estimator of $\tailmeasurestar(H)$ can be defined as
\begin{align}\label{eq:sliding-block-estimator-nonfeasible}
\tedclustersl(H):=\frac{1}{q_n \dhinterseq \pr(\norm{\bsX_0}>\tepseq)}\sum_{i=0}^{q_n-1}H\left(\bsX_{i+1,i+\dhinterseq}/\tepseq\right)\;.
\end{align}
The above estimator is not feasible, since it involves an unspecified sequence $\{\tepseq\}$ and the tail of $\norm{\bsX_0}$. Thus, in \eqref{eq:sliding-block-estimator-nonfeasible}
we replace $q_n\pr(\norm{\bsX_0}>\tepseq)$ with its empirical estimate $\sum_{j=1}^{q_n}\ind{\norm{\bsX_j}>\tepseq}$ to obtain a quasi-feasible estimator
\begin{align*}
\TEDclusterrandomaltsl_{n,\dhinterseq}(H)=\frac{1}{\dhinterseq}\frac{1}{\sum_{j=1}^{q_n}\ind{\norm{\bsX_j}>\tepseq}}
\sum_{i=0}^{q_n-1}H\left(\bsX_{i+1,i+\dhinterseq}/\tepseq\right)\;.
\end{align*}
Likewise, let $\statinterseqn$ be an intermediate sequence of integers, \ie\ $\lim_{n\to\infty}\statinterseqn=\infty$, $\lim_{n\to\infty}\statinterseqn/n=0$. Define $\tepseq$ by $\statinterseqn=n\pr(\norm{\bsX_0}>\tepseq)$.
Replacing $\tepseq$ with $\orderstat[\norm{\bsX}]{n}{n-\statinterseqn}$ and noting that (assuming for simplicity that there are not ties in the data)
$$
\sum_{j=1}^{n}\ind{\norm{\bsX_j}>\orderstat[\norm{\bsX}]{n}{n-\statinterseqn}}=\statinterseqn\;,
$$
we obtain a feasible estimator of $\tailmeasurestar(H)$ given in \eqref{eq:sliding-block-estimator-feasible}.
\subsection{Weak dependence assumptions}
For asymptotic normality, we need to strengthen the anticlustering condition \ref{eq:conditiondh}.
\begin{definition}\label{def:conditionS}
Condition
  \ref{eq:conditionS} holds if
for all $s,t>0$
\begin{align}
    \label{eq:conditionS}
    \lim_{m\to\infty} \limsup_{n\to\infty} \frac{1}{\pr(\norm{\bsX_0}>\tepseq)}
    \sum_{j=m}^{\dhinterseq} \pr(\norm{\bsX_0}>\tepseq s,\norm{\bsX_j}>\tepseq t) = 0 \;
    . \tag{$\conditiondhsumpaper{\dhinterseq}{\tepseq}$}
  \end{align}
\end{definition}
This condition implies that $\sum_{j\in\Zset}\pr(\norm{\bsY_j}>1)<\infty$. The latter series appears explicitly in the statement for the limiting variance.

Dependence in $\sequence{\bsX}$ will be controlled by the $\beta$-mixing rates $\{\beta_n\}$.  Recall \ref{eq:rnbarFun0}.
Let $\{\ell_n\}$ be a sequence of integers such that $\lim_{n\to\infty}\ell_n=\infty$ and $\lim_{n\to\infty} \ell_n/\dhinterseq=0$.
\begin{definition}\label{def:condition-beta-r-l-sl}
Condition $\beta(\dhinterseq)$ holds if:
\begin{enumerate}
\item $\beta_j=O(j^{-\nu})$, $\nu>1$ and $\lim_{n\to\infty}\dhinterseq^{1+\nu}/n=+\infty$; and
\item there exists $\delta>0$ such that
$\lim_{n\to\infty}\dhinterseq^{\nu-\delta}\pr(\norm{\bsX_0}>\tepseq)=+\infty$.
\end{enumerate}
\end{definition}
From the basic assumptions on the time series, we have $\lim_{n\to\infty}\dhinterseq/n=0$. Thus, $\nu$ has to be big enough.
The above mixing condition is clearly satisfied for time series with geometric mixing rates since then $\nu$ can be chosen arbitrarily large.
\subsection{Main result}

Let $\tepcluster$ be the Gaussian process on $L^2(\tailmeasurestar)$ with covariance
\begin{align*}
   \cov(\tepcluster(H),\tepcluster(\widetilde{H})) =  \tailmeasurestar(H\widetilde{H}) \; .
 \end{align*}
Recall that for a functional $H:(\Rset^d)^\Zset\to\Rset$ and $s>0$ we define $H_s(\bsx)=H(\bsx/s)$.

The main result of this paper is \Cref{thm:sliding-block-clt-1}, the asymptotic normality of the appropriately normalized estimator
$\TEDclusterrandomsl_{n,\dhinterseq}(H)$. The limiting variance agrees with the one for the disjoint blocks estimator; cf. \cite{drees:rootzen:2010} and \cite[Chapter 10]{kulik:soulier:2020}.
\begin{theorem}\label{thm:sliding-block-clt-1}
Let $\sequence{\bsX}$ be a stationary, regularly varying $\Rset^d$-valued time series. Assume
  that~\ref{eq:rnbarFun0},
  $\beta(\dhinterseq)$, \ref{eq:conditionS} hold.
  Fix $0<s_0<1<t_0<\infty$. Let $H:(\Rset^d)^\Zset\to \Rset$ be a
    shift-invariant measurable map such that the class $\{H_s:s\in [s_0,t_0]\}$ is linearly ordered. Assume moreover that
  \begin{subequations}
    \begin{align}
      \label{eq:bias-cluster-clt-aa}
      \lim_{n\to\infty}\sqrt{\statinterseqn}  \sup_{s\in [s_0,t_0]} |\esp[\tedclustersl(\exc_s)] -\tailmeasurestar(\exc_s)|
      &  = 0 \; , \\                                                                                                                                     \label{eq:bias-cluster-clt-func}
      \lim_{n\to\infty}\sqrt{\statinterseqn}  \sup_{s\in [s_0,t_0]} |\esp[\tedclustersl(H_s)] -\tailmeasurestar(H_s)| & = 0 \; .
    \end{align}
  \end{subequations}
If $H\in\mca$, then
\begin{align}\label{eq:clt-estimator-1}
\sqrt{\statinterseqn}\left\{\TEDclusterrandomsl_{n,\dhinterseq}(H)-\tailmeasurestar(H)\right\}\convdistr
\tepcluster(H-\tailmeasurestar(H)\exc)\;.
\end{align}
If moreover \ref{eq:AN-small} is satisfied, then \eqref{eq:clt-estimator-1} holds for $H\in\mcb$.
\end{theorem}
\begin{remark}\label{rem:limiting-variance}
  The limiting distribution is centered Gaussian with variance (cf. \Cref{exer:expression-tailmeasurestar-HH'}):
  \begin{align*}
    \tailmeasurestar& (\{H-\tailmeasurestar(H)\exc\}^2) = \tailmeasurestar(H^2) - 2\tailmeasurestar(H)\tailmeasurestar(H\exc) + (\tailmeasurestar(H))^2 \tailmeasurestar(\exc^2) \\
    & =  \tailmeasurestar(H^2) - 2\tailmeasurestar(H)\esp[H(\bsY)] + (\tailmeasurestar(H))^2 \sum_{j\in\Zset} \pr(\norm{\bsY_j}>1) \; .
  \end{align*}
   \remarkend
\end{remark}
\subsection{Examples}
\begin{example}[Extremal index]
  \label{xmpl:blocks-extremal_index}
  For $H(\bsx)=\ind{\bsx^*>1}$ we have $\tailmeasurestar(H)=\tailmeasurestar(H^2)=\canditheta$.
  The data-based estimator \eqref{eq:sliding-block-estimator-feasible}
  is asymptotically normal with mean zero and the limiting variance is
  \begin{align*}
    \tailmeasurestar(\{H-\tailmeasurestar(H)\exc\}^2)
    = \canditheta^2 \sum_{j\in\Zset}\pr(Y_j>1)-\canditheta\;.
  \end{align*}
See \Cref{sec:comments} for a discussion on the existing results.
   \exampleend
\end{example}

\begin{example}[Cluster size distribution]
  \label{xmpl:clt-cluster-size-dist}
  Consider the situation from \Cref{xmpl:clustersizedistribution}.
 The limiting distribution is centered normal with the variance
  \begin{align*}
    \pi(m)\left( 1 - 2\pr(\exc(\bsY)=m) \right)  + \pi^2(m) \sum_{j\in\Zset} \pr(Y_j>1) \; .
  \end{align*}
    \exampleend
\end{example}
\begin{example}[Stop-loss index]
  \label{xmpl:clt-stoploss}
  Consider the stop-loss index $\stoploss(\eta)$
  introduced in \Cref{xmpl:stop-loss}.
  The limiting distribution is centered normal with the variance
  \begin{align*}
    \stoploss(\eta)  \left(1- 2 \pr\left(\sum\nolimits_{j\in\Zset}(Y_j-1)_+>\eta\right)\right)  + \stoploss^2(\eta)\sum_{j\in\Zset} \pr(Y_j>1) \; .
  \end{align*}
 \exampleend
\end{example}
\begin{example}[Large deviations index]
  \label{xmpl:large-deviation-clt}
  We continue with the situation from
  \Cref{xmpl:large-deviation}.
  The limiting distribution is centered Gaussian with  variance
  \begin{align*}
    \largedeviation  -2 \largedeviation\pr\left(\left(\sum_{j\in\Zset} Y_j\right)_+>1\right)
    + \largedeviation^2 \sum_{j\in\Zset} \pr(\norm{Y_j}>1) \; .
 \end{align*}
  \exampleend
\end{example}

\section{Comments and extensions}\label{sec:comments}
\subsection{Non-linearly ordered function classes}\label{sec:nonlinear}
The linear ordering in \Cref{thm:sliding-block-clt-1} seems to be quite restrictive.
It can be replaced with an assumption that the function class is of VC-type (see \cite{drees:rootzen:2010}), or can be approximated by VC-classes. See \cite[Lemma A.3]{bilayi:kulik:soulier::2018}, \cite{drees:knezevic:2019} and \cite[Appendix C.4]{kulik:soulier:2020}.

\subsection{Existing results}
We discuss the existing results.
For the sake of clarify, we consider univariate, \nonnegative, regularly varying time series with the marginal distribution $F$. \

\paragraph{PoT approach.} In \cite{drees:neblung:2020} the authors study asymptotic normality of the sliding blocks estimators in a general setting. They show that the limiting variance of such estimators does not exceed the one for the disjoint blocks estimators. For the extremal index they found the variances to be equal. As in this paper,
they use the threshold $\tepseq$ such as in \ref{eq:rnbarFun0}.  The results in \cite{drees:rootzen:2010} and \cite[Chapters 9-10]{kulik:soulier:2020} (disjoint blocks) as well as in \cite{drees:neblung:2020} and in the current paper fit into Peak Over Threshold (PoT) framework.

In particular, consider the disjoint blocks estimator of the extremal index,
\begin{align*}
\widetilde{\canditheta}_{n,1}=\widetilde{\canditheta}_{n,1}(x)=\frac{\sum_{i=1}^{m_n} \ind{\bsX_{(i-1)\dhinterseq+1,i\dhinterseq}^*>x}}{\sum_{j=1}^{m_n\dhinterseq}\ind{X_j>x}}\;.
\end{align*}
In \cite[Example 10.4.2]{kulik:soulier:2020} we calculated the limiting variance of $\widetilde{\canditheta}_{n,1}(\tepseq)$ to be $\sigma_1^2=-\canditheta+\canditheta^2\sum_{j\in\Zset}\pr(Y_j>1)$. This is in agreement with
Corollary 4.6 in~\cite{hsing:1991} (where the variance $\sigma_1^2$ is given in a complicated form).
We can see that $\sigma_1^2$ agrees with the one limiting variance for the sliding blocks estimator in \Cref{xmpl:blocks-extremal_index}. The blocks
estimator $\widetilde{\canditheta}_{n,1}(\tepseq)$ is also considered in \cite{smith:weissman:1994} and \cite{weissman:novak:1998}.

\paragraph{Block maxima framework.}
One can also use the threshold $c_{\dhinterseq}$ given by \begin{align}\label{eq:threshold-2}
r_n \tail{F}(c_{r_n})\to 1\;.
\end{align}

We are not aware of the asymptotic theory for $\widetilde{\canditheta}_{n,1}(c_{\dhinterseq})$. However, using \cite[Theorem 4.2]{robert:segers:ferro:2009} and the delta method we can compare the variances of $\widetilde{\canditheta}_{n,1}(\tepseq)$ and $\widetilde{\canditheta}_{n,1}(c_{\dhinterseq})$:
\begin{align*}
\sigma_1^2=-\canditheta+\canditheta^2\sum_{j\in\Zset}\pr(Y_j>1) \ \ \mbox{\rm vs. }
\sigma_3^2:=\rme^{-\canditheta}(1-\rme^{-\canditheta})-2\canditheta\rme^{-\canditheta}+ \canditheta^2\sum_{j\in\Zset}\pr({Y_j}>1)\;.
\end{align*}
Thus, the estimator $\widetilde{\canditheta}_{n,1}(\tepseq)$ has a smaller variance than $\widetilde{\canditheta}_{n,1}(c_{r_n})$.

In the following discussion, we will use the threshold \eqref{eq:threshold-2}.
In \cite{robert:segers:ferro:2009} the authors consider another disjoint blocks estimator of the extremal index, motivated by
the approximation $\log(1-x)\sim x$ ($x\to 0$).
Also, the corresponding sliding blocks estimator is considered. It is shown that the sliding blocks one yields a smaller asymptotic variance.

In \cite{bucher:segers:2018disjoint,bucher:segers:2018sliding} the authors estimate the parameters $(\alpha,\sigma)$ of the Fr\'{e}chet distribution stemming from the limiting behaviour of the maxima. Disjoint blocks yield a larger variance than sliding blocks.
Similarly, in \cite{berghaus:bucher:2018} the authors use the blocking method to estimate the extremal index and again the sliding block estimator is more efficient.

The estimator $\widetilde{\canditheta}_{n,1}(c_{\dhinterseq})$ as well as the ones in
\cite{robert:segers:ferro:2009} and \cite{bucher:segers:2018disjoint,bucher:segers:2018sliding} can be thought of as the application of the block maxima method. Indeed, the threshold $c_{\dhinterseq}$ is the normalizing sequence for the limiting distribution of maxima. In the context of the latter two papers, $\bsX_{1,\dhinterseq}^*/\sigma_{\dhinterseq}$ converges in distribution to a standard Fr\'{e}chet random variable with tail index $\alpha$ (denoted by $Z$). On the other hand, for $\xi\in (0,1)$, the pair
\begin{align*}
(\bsX_{1,\dhinterseq}^*/\sigma_{\dhinterseq}, \bsX_{1+[\xi\dhinterseq],\dhinterseq+[\xi\dhinterseq]}^*/\sigma_{\dhinterseq})
\end{align*}
converges in distribution to a dependent random vector $(Z_1,Z_2)$ with Fr\'{e}chet marginals and parametrized by $\xi\in (0,1)$. See \cite[Lemma 5.1]{bucher:segers:2018sliding}. Consider now for $f:\Rset\to\Rset$
\begin{align*}
&\mathbb{G}_n^{(BS)}(f)=\sqrt{m_n}
\left\{m_n^{-1}\sum_{j=1}^{m_n}f\left(\frac{\bsX_{(j-1)\dhinterseq+1,j\dhinterseq}^*}{\sigma_{\dhinterseq}}\right)-\esp[f(Z)]\right\}\;,\\
&
\mathbb{F}_n^{(BS)}(f)=\sqrt{m_n}
\left\{q_n^{-1}\sum_{i=1}^{q_n}f\left(\frac{\bsX_{i,i+\dhinterseq-1}^*}{\sigma_{\dhinterseq}}\right)-\esp[f(Z)]\right\}\;.
\end{align*}
The aforementioned convergence gives the limiting variance. For the disjoint blocks empirical process the limiting variance is $\var(f(Z))$, while for the sliding blocks one it becomes (cf. Lemma 5.3 in \cite{bucher:segers:2018sliding})
\begin{align*}
C(f):=2\int_0^1\cov_\xi (f(Z_1),f(Z_2))\rmd \xi\;.
\end{align*}
In the context of our paper, if we choose $f(z)=\ind{z>1}$, then we can evaluate:
\begin{align*}
\var(f(Z))=\exp(-1)-\exp(-2)>C(f)=2\exp(-1)-4\exp(-2)\;.
\end{align*}
In the PoT framework considered in our paper, both the disjoint blocks and the sliding blocks empirical processes yield the limiting variance $\tailmeasurestar(H)$.

In summary:
\begin{itemize}
\item
\textbf{The PoT method, as proven in this paper, gives the same limiting behaviour for both disjoint and sliding blocks estimators.}
\item  The situation seems to be different in case of the block maxima method, at least for the inference problems considered up to date.
\item One can argue that
 the blocks maxima method is restricted to estimation of the parameters of the limiting distribution of maxima (the tail index, the extremal index) and is rather hard to see how the method can be employed to other cluster indices.
\end{itemize}

\subsection{Open questions}
\begin{itemize}
\item For the sliding blocks estimators, obtain consistency under minimal conditions (that is, without relying on $\beta$-mixing). In \cite[Chapter 10]{kulik:soulier:2020} we obtain consistency of the disjoint blocks estimators for time series that can be approximated by $m$-dependent sequences, including long memory ones.
\item Extend \Cref{thm:sliding-block-clt-1} to unbounded functionals $H$. The method of the proof presented in the paper should be applicable, however, some substantial modifications may be needed. Certainly, more restrictive conditions will need to be implemented.
\item In view of the behaviour of $\widetilde{\canditheta}_{n,1}(\tepseq)$ and $\widetilde{\canditheta}_{n,1}(c_{\dhinterseq})$, it would be interesting to know if (whenever possible) the PoT method always gives a smaller variance than the block maxima ones.
\end{itemize}

\section{Proofs}\label{sec:proofs}
In \Cref{sec:convergence-of-clusters-proofs} we show that \eqref{eq:thelimitwhichisnolongercalledbH} holds for $H\in \mca\cup\mcb$. The proofs in that section stem from \cite{kulik:soulier:2020}. The results from  \Cref{sec:convergence-of-clusters-proofs} are extended in \Cref{sec:convergence-of-clusters-proofs-cov} to covariance of clusters. In \Cref{sec:empirical-process} we introduce the empirical process of sliding blocks and state its functional convergence. The proof of the latter is separated into several parts. First, in
\Cref{sec:covariance-empirical-process} we derive the limiting covariance of the empirical process of sliding blocks. Next, in \Cref{sec:fidi} we prove the \fidi\ convergence.  Asymptotic continuity is dealt with in
\Cref{sec:tightness}. We conclude the proof in \Cref{sec:proof-conclusion}.

\subsection{Consequences of the mixing assumption}
Since $\ell_n$ can be chosen as $\dhinterseq^{1-\delta}$ ($\delta>0$) with $\delta$ arbitrarily close to zero, $\beta(\dhinterseq)$ gives:
\begin{subequations}
\begin{align}
&\lim_{n\to\infty}\frac{n}{\dhinterseq}\beta_{\dhinterseq}=0\;,\label{eq:mixing-rates-0}\\
&\lim_{n\to\infty}
\frac{1}
{\dhinterseq\pr(\norm{\bsX_0}>\tepseq)}
\sum_{j=\ell_n}^{\infty}\beta_{j}=0\;,\label{eq:mixing-rates-1a}\\
&\lim_{n\to\infty}\frac{1}{\dhinterseq \pr(\norm{\bsX_0}>\tepseq)}\sum_{j=1}^\infty \beta_{j\dhinterseq}=0\;.\label{eq:mixing-rates-7}
\end{align}
\end{subequations}
We recall the covariance inequality for bounded, beta-mixing random variables (in fact, the inequality holds for $\alpha$-mixing).
Let $\beta({\mcf_1},{\mcf_2})$ be the $\beta$-mixing coefficient between two sigma fields.
Then (\cite{ibragimov:1962})
\begin{align}\label{eq:davydov-2}
|\cov(H(Z_1),H(Z_2))|\leq \constant\ \|H\|_{\infty} \|\widetilde{H}\|_{\infty}\beta(\sigma(Z_1),\sigma(Z_2))\;.
\end{align}
In \eqref{eq:davydov-2} the constant $\constant$ does not depend on $H,\widetilde{H}$.

\subsection{Convergence of cluster measure}\label{sec:convergence-of-clusters-proofs}

\begin{proof}[Proof of \Cref{theo:cluster-RV}]
  Since $H$ has a support separated from zero, there exists $\epsilon>0$ such that $H(\bsx)=0$ if
  $\bsx^*\leq\epsilon$. Applying its shift invariance and the stationarity,
  we obtain
  \begin{align*}
    \tailmeasurestar_{n,\dhinterseq}(H)
    & = \frac1{\dhinterseq\pr(\norm{\bsX_0}>\tepseq)}  \sum_{j=1}^{\dhinterseq}
      \esp\left[ H(\tepseq^{-1} \bsX_{1,\dhinterseq}) \ind{\bsX_{1,j-1}^*\leq \tepseq \epsilon} \ind{\norm{\bsX_j}>\tepseq \epsilon}\right] \\
    & = \frac{\pr(\norm{\bsX_0}>\tepseq \epsilon)} {\pr(\norm{\bsX_0}>\tepseq)} \frac1{\dhinterseq}  \sum_{j=1}^{\dhinterseq}
      \esp\left[ H(\tepseq^{-1} \bsX_{1-j,\dhinterseq-j})
      \ind{\bsX_{1-j,-1}^*\leq \tepseq \epsilon} \mid \norm{\bsX_0}>\tepseq \epsilon \right] \\
    & =  \frac{\pr(\norm{\bsX_0}>\tepseq \epsilon)}  {\pr(\norm{\bsX_0}>\tepseq)} \int_0^1 g_n(v) \rmd v \; ,
  \end{align*}
  with
  \begin{align*}
    g_n(v) = \esp\left[ H(\tepseq^{-1} \bsX_{1-[\dhinterseq v],\dhinterseq-[\dhinterseq v]})
      \ind{\bsX_{1-[\dhinterseq v],-1}^*\leq \tepseq \epsilon} \mid \norm{\bsX_0}>\tepseq \epsilon \right] \; .
  \end{align*}
  By \Cref{lem:tailprocesstozero},
  $\lim_{n\to\infty} g_n(v)=\esp[H(\epsilon\bsY)\ind{\bsY_{-\infty,-1}^*\leq1}]$ for each
  $v\in(0,1)$.  Moreover, the sequence $g_n$ is uniformly bounded, thus by dominated convergence,
  regular variation of $\norm{\bsX_0}$ and \eqref{eq:relation-cluster-Y-Q-Theta-epsilon}, we obtain
  \begin{align*}
    \lim_{n\to\infty} \tailmeasurestar_{n,\dhinterseq}(H)
    & = \epsilon^{-\alpha} \esp[H(\epsilon \bsY)\ind{\bsY_{-\infty,-1}^*\leq1}] =  \tailmeasurestar(H) \; .
  \end{align*}
\end{proof}
\begin{proof}[Proof of \Cref{lem:summability-Q}]
  By \Cref{theo:cluster-RV} and \eqref{eq:seq-Q}, we have
  \begin{multline}
    \lim_{n\to\infty} \frac{\pr(\sum_{j=1}^{\dhinterseq}
      \norm{\bsX_j}\ind{\norm{\bsX_j}>\epsilon\tepseq }>\tepseq )}{\dhinterseq\pr(\norm{\bsX_0}>\tepseq )} \\
    = \canditheta \int_0^\infty \pr\left(z\sum_{j\in\Zset} \norm{\bsQ_j}
      \ind{z\norm{\bsQ_j}>\epsilon}>1\right) \alpha z^{-\alpha-1} \rmd z \;
    . \label{eq:epsilon-sumQ-R}
  \end{multline}
  By monotone convergence, the right hand side converges as $\epsilon\to0$ to
  $\canditheta \esp\left[ \left(\sum_{j\in\Zset} |\bsQ_j|\right)^\alpha \right]$.

  Consider the function
  \begin{align*} g(\zeta)=\canditheta\int_0^\infty \pr\left( z \sum_{j\in\Zset} |\bsQ_j| \ind{z|\bsQ_j| >
      \zeta}>1 \right) \alpha z^{-\alpha-1} \rmd z\;. \end{align*}
  It increases when $\zeta$ decreases to zero and its limit is
  $\canditheta\esp\left[ \left(\sum_{j\in\Zset} |\bsQ_j|\right)^\alpha \right] $. To prove
  that this quantity is finite, it suffices to prove that the function $g$ is bounded.  Fix $\epsilon>0$ and
  $\eta\in(0,1)$. By \ref{eq:AN-small}, there exists $\zeta$ such that
  \begin{align*}
    \limsup_{n\to\infty} \frac{\pr(\sum_{j=1}^{\dhinterseq}\norm{\bsX_j}\ind{\norm{\bsX_j}\leq \zeta\tepseq } > \eta
    \tepseq )} {\dhinterseq\pr(\norm{\bsX_0}>\tepseq )} \leq \epsilon \; .
  \end{align*}
  Fix $\zeta'<\zeta$. Starting from \eqref{eq:epsilon-sumQ-R} and applying
  \ref{eq:AN-small}, we obtain
  \begin{align*}
    0 & \leq g(\zeta')=\canditheta \int_0^\infty \pr\left( z \sum_{j\in\Zset} |\bsQ_j| \ind{z|\bsQ_j| >
        \zeta'}>1 \right) \alpha z^{-\alpha-1} \rmd z \\
      & = \lim_{n\to\infty}
        \frac{\pr(\sum_{j=1}^{\dhinterseq}|\bsX_j|\ind{|\bsX_j|>\tepseq \zeta'}>\tepseq )}{\dhinterseq\pr(|\bsX_0|>\tepseq )}
    \\
      & = \lim_{n\to\infty} \frac{\pr\left(\sum_{j=1}^{\dhinterseq}|\bsX_j|\ind{|\bsX_j|>\tepseq \zeta} +
        |\bsX_j|\ind{\tepseq \zeta\geq|\bsX_j|>\epsilon
        \tepseq \zeta'}>\tepseq \right)} {\dhinterseq\pr(|\bsX_0|>\tepseq )}
    \\
      & \leq \limsup_{n\to\infty} \frac{\pr(\sum_{j=1}^{\dhinterseq} \norm{\bsX_j}
        \ind{\norm{\bsX_j} \leq \tepseq \zeta} > \eta \tepseq )} {\dhinterseq\pr(\norm{\bsX_0}>\tepseq )} \\
      & \phantom{ ==== } + \lim_{n\to\infty} \frac{\pr(\sum_{j=1}^{\dhinterseq} \norm{\bsX_j} \ind{\norm{\bsX_i} >
        \tepseq \zeta} > (1-\eta) \tepseq )} {\dhinterseq\pr(\norm{\bsX_0} >  \tepseq )} \\
      & \leq \epsilon +  \canditheta \int_0^\infty \pr\left( z \sum_{j\in\Zset} \norm{\bsQ_j} \ind{z\norm{\bsQ_j} > \zeta} > 1-\eta
        \right) \alpha z^{-\alpha-1} \rmd z  \leq \epsilon + \canditheta\zeta^{-\alpha}  \; .
  \end{align*}
  The latter bound holds since the probability inside the integral is zero if $z \leq \zeta$
  since~$|\bsQ_j|\leq 1$ for all $j$.  This proves that the function $g$ is bounded in a
  neighbourhood of zero as claimed.

  By Condition~\ref{eq:AN-small}, we finally obtain
  \begin{align*}
    \lim_{n\to\infty} & \frac{\pr\left(\sum_{j=1}^{\dhinterseq}|\bsX_j|>\tepseq \right)} {\dhinterseq\pr(|\bsX_0|>\tepseq )}  = \lim_{\epsilon\to0} \lim_{n\to\infty} \frac{\pr\left
                        (\sum_{j=1}^{\dhinterseq}|\bsX_j|\ind{|\bsX_{j}|>\tepseq \epsilon}>\tepseq \right)}
                        {\dhinterseq\pr(|\bsX_0|>\tepseq )}    \\
                      & = \lim_{\epsilon\to0} \canditheta \int_0^\infty \pr\left( z \sum_{j\in\Zset} |\bsQ_j|
                        \ind{z|\bsQ_j|>\epsilon} > \epsilon \right) \alpha z^{-\alpha-1} \rmd z  = \canditheta
                        \esp\left[ \left(\sum_{j\in\Zset} |\bsQ_j|\right)^\alpha \right] \; .
  \end{align*}
\end{proof}
\begin{proof}[Proof of \Cref{theo:limit-lone-Q}]
  For $\epsilon>0$, we define the truncation operator $T_\epsilon$ by
  \begin{align}
    \label{eq:def-truncation-operator}
    T_\epsilon(\bsx) = \{\bsx_j\indicsmall{|\bsx_j|>\epsilon},j\in\Zset\} \; .
  \end{align}
  The operator $T_\epsilon$ is continuous \wrt\ the uniform norm at every $\bsx\in\lzero$ such that
  $|\bsx_j|\ne\epsilon$ for all $j\in\Zset$.

  Fix $\eta\in(0,1)$ and $\zeta>0$. Let $L_K$ be as in~\eqref{eq:hypo-K} and choose $\epsilon>0$ such
  that
  \begin{align*}
    \limsup_{n\to\infty} \frac{\pr(\sum_{j=1}^{\dhinterseq} \norm{\bsX_{j}}\ind{\norm{\bsX_j}\leq\epsilon
        \tepseq } > \eta \tepseq /L_K)} {\dhinterseq\pr(\norm{\bsX_0}>\tepseq )} \leq \zeta \; .
  \end{align*}
  Set $K_\epsilon = K\circ
  T_\epsilon$. Applying assumption~\eqref{eq:hypo-K}, we obtain
  \begin{align*}
    & \frac{\pr(K(\bsX_{1,\dhinterseq}/\tepseq )>1)} {\dhinterseq\pr(\norm{\bsX_0}>\tepseq )} \\
    & \leq
      \frac{\pr(K_\epsilon(\bsX_{1,\dhinterseq}/\tepseq ) > 1-\eta)} {\dhinterseq\pr(\norm{\bsX_0}>\tepseq )} +
      \frac{\pr(|K(\bsX_{1,\dhinterseq}/\tepseq ) - K_\epsilon(\bsX_{1,\dhinterseq}/\tepseq )| > \eta)} {\dhinterseq\pr(\norm{\bsX_0}>\tepseq )} \\
    & \leq \frac{\pr(K_\epsilon(\bsX_{1,\dhinterseq}/\tepseq ) > 1-\eta)} {\dhinterseq\pr(\norm{\bsX_0}>\tepseq )} +
      \frac{\pr(\sum_{i=1}^{\dhinterseq} \norm{\bsX_{j}}\ind{\norm{\bsX_j}\leq\epsilon \tepseq } > \eta \tepseq /\constant)}
      {\dhinterseq\pr(\norm{\bsX_0}>\tepseq )} \; .
  \end{align*}
  Applying \Cref{theo:cluster-RV} to $K_\epsilon$, this yields
  \begin{align*}
    \limsup_{n\to\infty} \frac{\pr(K(\bsX_{1,\dhinterseq}/\tepseq )>1)} {\dhinterseq\pr(\norm{\bsX_0}>\tepseq )} & \leq
    \limsup_{n\to\infty} \frac{\pr(K_\epsilon(\bsX_{1,\dhinterseq}/\tepseq ) > 1-\eta)} {\dhinterseq\pr(\norm{\bsX_0}>\tepseq )} + \zeta \\
    & = \int_0^\infty \pr(K_\epsilon(z\bsQ)>1-\eta) \alpha z^{-\alpha-1} \rmd z + \zeta \; .
  \end{align*}
  Similarly,
  \begin{align*}
    \liminf_{n\to\infty} \frac{\pr(K(\bsX_{1,\dhinterseq}/\tepseq ) > 1)} {\dhinterseq\pr(\norm{\bsX_0}>\tepseq )} & \geq
    \int_0^\infty \pr(K_\epsilon(z\bsQ)>1+\eta) \alpha z^{-\alpha-1} \rmd z - \zeta \; .
  \end{align*}
  Since $K(\bszero)=0$, \eqref{eq:hypo-K} implies that $|K(\bsx)|\leq \constant \sum_{j\in\Zset}\norm{\bsx_j}$, thus for all $y>0$,
  \begin{align*}
    \pr(K_\epsilon(z\bsQ)>y) \leq \pr\left(\sum_{j\in\Zset} z\norm{\bsQ_j}>y/\constant \right)
  \end{align*}
  and the latter quantity is integrable (as a function of $z$) with respect to
  $\alpha z^{-\alpha-1}\rmd z$ in view of~\ref{eq:AN-small} and \Cref{lem:summability-Q}.
  By bounded convergence, this yields
  \begin{align*}
    \lim_{\epsilon\to0} \int_0^\infty \pr(K_\epsilon(z\bsQ)>y) \alpha z^{-\alpha-1} \rmd z =
    \int_0^\infty \pr(K(z\bsQ)>y) \alpha z^{-\alpha-1} \rmd z \; .
  \end{align*}
  Altogether, we obtain
  \begin{multline*}
    \int_0^\infty \pr(K_\epsilon(z\bsQ)>1+\eta) \alpha z^{-\alpha-1} \rmd z -\zeta \leq
    \liminf_{n\to\infty} \frac{\pr(K(\bsX_{1,\dhinterseq}/\tepseq )>1)} {\dhinterseq\pr(\norm{\bsX_0}>\tepseq )} \\
    \leq \limsup_{n\to\infty} \frac{\pr(K(\bsX_{1,\dhinterseq}/\tepseq )>1)} {\dhinterseq\pr(\norm{\bsX_0}>\tepseq )}  \leq
    \int_0^\infty \pr(K_\epsilon(z\bsQ)>1-\eta) \alpha z^{-\alpha-1} \rmd z+\zeta \; .
  \end{multline*}
  Since $\zeta$ and $\eta$ are arbitrary, this finishes the proof. 
\end{proof}
\subsection{Covariance of clusters}\label{sec:convergence-of-clusters-proofs-cov}
We consider the limit
\begin{align*}
&\lim_{n\to\infty}\frac{1}{\dhinterseq\pr(\norm{\bsX_0}>\tepseq)}\cov\left(H(\bsX_{1,\dhinterseq}/\tepseq)
,\widetilde{H}(\bsX_{1+h,\dhinterseq+h}/\tepseq)\right)
\end{align*}
for different choices of $h$, possibly depending on $n$.
Under the conditions of \Cref{theo:cluster-RV}, if moreover \ref{eq:rnbarFun0} holds, the above limit is the same as
\begin{align*}
\lim_{n\to\infty}\frac{1}{\dhinterseq\pr(\norm{\bsX_0}>\tepseq)}\esp\left[H(\bsX_{1,\dhinterseq}/\tepseq)
\widetilde{H}(\bsX_{1+h,\dhinterseq+h}/\tepseq)\right]\;.
\end{align*}
Thus, we impose \ref{eq:rnbarFun0} and switch freely between $\esp$ and $\cov$ whenever suitable.

\subsubsection{Uniform convergence of cluster measure}
In \Cref{theo:cluster-RV,theo:limit-lone-Q} we proved \eqref{eq:thelimitwhichisnolongercalledbH} for $H\in\mca\cup\mcb$. We note further that
if \eqref{eq:thelimitwhichisnolongercalledbH} holds for $H$ and $\widetilde{H}$, then it also holds for any linear combination of both functions. To deal with asymptotic normality, we need \eqref{eq:thelimitwhichisnolongercalledbH} to hold uniformly over a subclass of functions.
With this in mind, we introduce two additional classes of functions. First, we recall that
for a class $\mcg$ of functions $H:(\Rset^d)^\Zset\to \Rset$ its envelope is
\begin{align*}
\envelope[G](\bsx)=\sup_{H\in\mcg}|H(\bsx)|\;, \ \ \bsx\in(\Rset^d)^\Zset \;.
\end{align*}
\begin{definition}
$\widetilde\mca \subseteq {\rm span}(\mca)$ (resp. $\widetilde\mcb \subseteq {\rm span}(\mcb)$) is a class of functions with a finite envelope such that
\begin{align}\label{eq:class-Atilde}
\lim_{n\to\infty}\sup_{H\in \widetilde\mca}\tailmeasurestar_{n,\dhinterseq}(|H|)<\infty
\end{align}
 (resp. $\lim_{n\to\infty}\sup_{H\in \widetilde\mcb}\tailmeasurestar_{n,\dhinterseq}(|H|)<\infty$) and that
 for each $H$ there exist functions $K_n^H:(\Rset^d)^{\ell_n}\to\Rset_+$ such that
      \begin{align}\label{neg-small-blocks-new}
        \left| H\left(\frac{\bsX_{1,r_n}}{\tepseq}\right) - H\left(\frac{\bsX_{1,r_n-\ell_n}}{\tepseq}\right) \right|
         \leq K_n^H(\bsX_{r_n-\ell_n+1,r_n})\;, \ \
         \lim_{n\to\infty} \sup_{H\in \widetilde\mca}\frac{\esp\left[K_{n}^H(\bsX_{1,\ell_n}) \right]}{\dhinterseq \pr(\norm{\bsX_0}>\tepseq)} = 0\;.
   \end{align}
\end{definition}
\begin{remark}
The uniform convergence condition \eqref{eq:class-Atilde} strengthens the statement of \Cref{theo:cluster-RV}. Conditions \eqref{eq:class-Atilde} and \eqref{neg-small-blocks-new} are needed for asymptotic equicontinuity of empirical cluster process to be introduced below.\remarkend
\end{remark}
\begin{remark}
We note that
$$
\lim_{n\to\infty} \frac{\esp\left[K_{n}^H(\bsX_{1,\ell_n}) \right]}{\dhinterseq \pr(\norm{\bsX_0}>\tepseq)} = 0
$$
for each $H\in\widetilde\mca\cup \widetilde\mcb$. Let us verify it for
$H\in \mcb$. We have
\begin{multline*}
    |\ind{K(\bsx_{1,\dhinterseq})>1}-\ind{K(\bsx_{1,\dhinterseq-\ell_n})>1}| \\
    = \ind{K(\bsx_{1,\dhinterseq})>1}\ind{K(\bsx_{1,\dhinterseq-\ell_n})\leq 1}
    + \ind{K(\bsx_{1,\dhinterseq})\leq 1}\ind{K(\bsx_{1,\dhinterseq-\ell_n})> 1} \; .
  \end{multline*}
  We consider the first pair of indicators in the last line. The events
  $\{K(\bsx_{1,\dhinterseq})>1\}$ and $\{K(\bsx_{1,\dhinterseq-\ell_n})\leq 1\}$ imply that there
  exists $s>0$ such that $K(\bsx_{1,\dhinterseq})-K(\bsx_{1,\dhinterseq-\ell_n})>s$.
Applying the same reasoning to the second pair of indicators, we have
\begin{align*}
|\ind{K(\bsx_{1,\dhinterseq})>1}-\ind{K(\bsx_{1,\dhinterseq-\ell_n})>1}|\leq 2\ind{\constant \sum_{j=\dhinterseq-\ell_n+1}^{\dhinterseq}\norm{\bsx_j}>s}\;.
\end{align*}
Since $\ell_n=o(\dhinterseq)$
\begin{align*}
\pr\left(\sum_{j=\dhinterseq-\ell_n+1}^{\dhinterseq}\norm{\bsX_j}>s \tepseq\right)=O(\ell_n\pr(\norm{\bsX_0}>\tepseq))=(\dhinterseq\pr(\norm{\bsX_0}>\tepseq))\;.
\end{align*}
In summary,
\eqref{neg-small-blocks-new} holds if the envelope function is in $\widetilde\mca\cup \widetilde\mcb$.
\remarkend
\end{remark}
\begin{remark}\label{rem:A-LinA-delta}
Let $\delta>0$. If $H$ is bounded then
\begin{align*}
\tailmeasurestar_{n,\dhinterseq}(|H|^{1+\delta})\leq \|H\|_\infty^{\delta}
\tailmeasurestar_{n,\dhinterseq}(|H|)
\end{align*}
and by the assumptions on the classes $\widetilde\mca$ and $\widetilde\mcb$,
$$
\lim_{n\to\infty}\sup_{H\in\widetilde\mca\cup\widetilde\mcb}\tailmeasurestar_{n,\dhinterseq}(|H|^{1+\delta})<\infty\;.
$$
\remarkend
\end{remark}
\begin{remark}
\label{rem:A-LinA}
Assume that~\ref{eq:conditiondh} holds.
Fix $0<s_0<t_0<\infty$.
Let $H\in\mca$ and recall that $H_s(\bsx)=H(\bsx/s)$. Assume that $\widetilde{\mca}:=\{H_s,s\in [s_0,t_0]\}$
is linearly ordered. Note that $\widetilde{\mca}\subset\mca$.
The envelope is $|H_{s_0}|\vee |H_{t_0}|\in\widetilde{\mca}$
hence \eqref{neg-small-blocks-new} holds.
Moreover,
$\sup_{s\in [s_0,t_0]}\tailmeasurestar_{n,\dhinterseq}(|H_s|)$ is achieved at $s_0$ or $t_0$. Likewise,
\begin{align*}
\lim_{n\to\infty}\sup_{s,t\in [s_0,t_0]}\tailmeasurestar_{n,\dhinterseq}(|H_s-H_t|)<\infty\;.
\end{align*}
The same applies to $H\in\mcb$ if additionally \ref{eq:AN-small} holds.
\remarkend
\end{remark}
\subsubsection{Conditional convergence}\label{sec:cond-convergence}
We consider conditional convergence of functions $H,\widetilde{H}$ acting on overlapping blocks.
\begin{lemma}\label{lem:cond-conv-H}
Assume that \ref{eq:conditiondh} holds. Let $h<\dhinterseq$,
$H,\widetilde{H}\in \mcl$ and \textcolor{black}{$\widetilde{H}(\bszero)=0$}.
Then
\begin{align*}
\lim_{n\to\infty}\esp[H(\bsX_{1,\dhinterseq}/\tepseq)\widetilde{H}(\bsX_{1+h,\dhinterseq+h}/\tepseq)\mid \norm{\bsX_0}>\tepseq]\\
=\left\{
\begin{array}{ll}
\esp[H(\bsY_{1,\infty})\widetilde{H}(\bsY_{1+h,\infty})] \;, & \mbox{\rm if } h\ \mbox{\rm fixed}\;,\\
0\;, & \mbox{\rm if } h=h_n\to\infty\;.
\end{array}
\right.
\end{align*}
and
\begin{align*}
\lim_{n\to\infty}\esp[H(\bsX_{-\dhinterseq,\dhinterseq}/\tepseq)\widetilde{H}(\bsX_{-\dhinterseq+h,\dhinterseq+h}/\tepseq)\mid \norm{\bsX_0}>\tepseq]=
\esp[H(\bsY)\widetilde{H}(\bsY)]\;.
\end{align*}
\end{lemma}
\begin{proof}
Since $H,\widetilde{H}$ are bounded, the first expectation of interest is dominated by
\begin{align*}
\|H\|_{\infty}\|\widetilde{H}\|_{\infty}\pr(\bsX_{1+h,\dhinterseq+h}^*>\tepseq\mid \norm{\bsX_0}>\tepseq)\;.
\end{align*}
Thus, the statement for $h=h_n\to\infty$ follows immediately from \ref{eq:conditiondh} (cf. the argument in the proof of \cite[Theorem 6.1.4]{kulik:soulier:2020}).

Now, let $h$ be fixed. Fix $r$. Since $H,\widetilde{H}$ are bounded Lipschitz continuous, we have by  \Cref{lem:tailprocesstozero},
\begin{align*}
&\lim_{n\to\infty}\esp[H(\bsX_{1,r}/\tepseq)\widetilde{H}(\bsX_{1+h,r+h}/\tepseq)\mid \norm{\bsX_0}>\tepseq]
=\esp[H(\bsY_{1,r})\widetilde{H}(\bsY_{1+h,r+h})]\;,\\
&\lim_{n\to\infty}\esp[H(\bsX_{-r,r}/\tepseq)\widetilde{H}(\bsX_{-r+h,r+h}/\tepseq)\mid \norm{\bsX_0}>\tepseq]
=\esp[H(\bsY_{-r,r})\widetilde{H}(\bsY_{-r+h,r+h})]\;.
\end{align*}
Since the tail process tends to zero under
  condition~\ref{eq:conditiondh}, it also holds that
  \begin{align*}
    &\lim_{r\to\infty}  \esp[H(\bsY_{1,r})\widetilde{H}(\bsY_{1+h,r+h})] = \esp[H(\bsY_{1,\infty})\widetilde{H}(\bsY_{1+h,\infty})]\;,  \\
    &\lim_{r\to\infty}  \esp[H(\bsY_{-r,r})\widetilde{H}(\bsY_{-r+h,r+h})] = \esp[H(\bsY)\widetilde{H}(\bsY)]=\esp[H(\bsY)\widetilde{H}(\bsY)] \; .
  \end{align*}
Indeed, considering  the first statement only we have
\begin{align*}
&\lim_{r\to\infty}\left|\esp[H(\bsY_{1,r})\widetilde{H}(\bsY_{1+h,r+h})] - \esp[H(\bsY_{1,\infty})\widetilde{H}(\bsY_{1+h,\infty})]\right|\\
&\leq \lim_{r\to\infty}\esp[\left|H(\bsY_{1,r})-H(\bsY_{1,\infty})\right||\widetilde{H}|(\bsY_{1+h,r+h})] \\ &\phantom{=}+\lim_{r\to\infty}\esp\left[|H|(\bsY_{1,\infty})\left|\widetilde{H}(\bsY_{1+h,r+h})-\widetilde{H}(\bsY_{1+h,\infty})\right|\right]\\
&\leq \|\widetilde{H}\|_{\infty}\lim_{r\to\infty}\left\{\esp\left[\left|H(\bsY_{1,r})-H(\bsY_{1,\infty})\right|\right]
+\esp\left[\left|\widetilde{H}(\bsY_{1+h,r+h})-\widetilde{H}(\bsY_{1+h,\infty})\right|\right]\right\}=0\;.
\end{align*}
To conclude, we only need to apply the triangular argument, that
  is to prove that
  \begin{align*}
    &\lim_{r\to\infty}\limsup_{n\to\infty}\\
    &
    \left| \esp[
    H(\bsX_{1,r}/\tepseq)\widetilde{H}(\bsX_{1+h,r+h}/\tepseq) -
    H(\bsX_{1,\dhinterseq}/\tepseq)\widetilde{H}(\bsX_{1+h,\dhinterseq+h}/\tepseq)  \mid \norm{\bsX_0}>\tepseq] \right| = 0 \; .
  \end{align*}
  Using again the fact that $H,\widetilde{H}$ are bounded, the conditional expectation is dominated by
  \begin{align}
  &\|\widetilde{H}\|_{\infty}\left|\esp[H(\bsX_{1,r}/\tepseq)-H(\bsX_{1,\dhinterseq}/\tepseq)\mid \norm{\bsX_0}>\tepseq]\right|\nonumber\\
  &+\|H\|_{\infty}\left|\esp[\widetilde{H}(\bsX_{1+h,r+h}/\tepseq)-\widetilde{H}(\bsX_{1+h,\dhinterseq+h}/\tepseq)\mid \norm{\bsX_0}>\tepseq]\right|\;.\label{eq:cond-conv-bound-1}
  \end{align}
  Fix $\epsilon>0$.  Since $H$ is Lipschitz continuous, applying condition~\ref{eq:conditiondh}
  yields
  \begin{multline*}
    \lim_{r\to\infty} \limsup_{n\to\infty}
    \left| \esp[H(\bsX_{1,r}/\tepseq) - H(\bsX_{1,\dhinterseq}/\tepseq) \mid \norm{\bsX_0}>\tepseq] \right| \\
    \leq \constant \left\{ \epsilon +     \lim_{r\to\infty} \limsup_{n\to\infty} \pr( \bsX_{r,\dhinterseq}^{*}>\epsilon
      \tepseq \mid \norm{\bsX_0}>\tepseq) \right\} = \constant \times \epsilon \; .
  \end{multline*}
  The same argument applies to \eqref{eq:cond-conv-bound-1}.
  Since $\epsilon$ is arbitrary, this concludes the proof.
\end{proof}

\subsubsection{Covariance of clusters: Disjoint blocks}
The first result is straightforward under the beta-mixing conditions.
\begin{lemma}[Disjoint blocks I]\label{lem:covariance-disjoint-blocks}
Assume that \ref{eq:conditiondh},~\ref{eq:rnbarFun0}, \eqref{eq:mixing-rates-1a} hold.
Then
\begin{align*}
&\lim_{n\to\infty}
\frac{1}{\dhinterseq\pr(\norm{\bsX_0}>\tepseq)}\sup_{\xi'>1}\sup_{H,\widetilde{H}\in\widetilde\mca\cup\widetilde\mcb}
\esp\left[H(\bsX_{1,\dhinterseq}/\tepseq)
\widetilde{H}(\bsX_{1+[\xi' \dhinterseq],\dhinterseq+[\xi'\dhinterseq]}/\tepseq)\right]=0\;.
\end{align*}
\end{lemma}
\begin{proof}[Proof of \Cref{lem:covariance-disjoint-blocks}]
Let $H,\widetilde{H}\in \widetilde\mca\cup\widetilde\mcb$. Then, using \eqref{neg-small-blocks-new},
\begin{align*}
&\left|\esp\left[H(\bsX_{1,\dhinterseq}/\tepseq)
\widetilde{H}(\bsX_{1+[\xi' \dhinterseq],\dhinterseq+[\xi'\dhinterseq]}/\tepseq)\right]\right|\\
&\leq \left|\esp\left[H(\bsX_{1,\dhinterseq-\ell_n}/\tepseq)
\widetilde{H}(\bsX_{1+[\xi' \dhinterseq],\dhinterseq+[\xi'\dhinterseq]}/\tepseq)\right]\right|
+ \|\widetilde{H}\|_{\infty} \esp\left[K_n^H(\bsX_{1,\ell_n})\right]\;
\end{align*}
and the latter term is $o(\dhinterseq\pr(\norm{\bsX_0}>\tepseq))$, uniformly over $\widetilde{\mca}\cup\widetilde{\mcb}$ by the assumption.

Using \eqref{eq:davydov-2} and \eqref{eq:mixing-rates-1a}, we have
\begin{align*}
&
\frac{\left|\cov\left(H(\bsX_{1,\dhinterseq-\ell_n}/\tepseq)
,\widetilde{H}(\bsX_{1+[\xi'\dhinterseq],\dhinterseq+[\xi'\dhinterseq]}/\tepseq)\right)\right| }{\dhinterseq\pr(\norm{\bsX_0}>\tepseq)}
\leq \|H\|_\infty\|\widetilde{H}\|_\infty \frac{\beta_{\ell_n+[(\xi'-1)\dhinterseq]}}{\dhinterseq\pr(\norm{\bsX_0}>\tepseq)}
\end{align*}
and the latter is $o(1)$
uniformly over the class of functions.
\end{proof}
We extend the above result to the excess functional $\exc_s(\bsx)=\sum_{j\in\Zset}\ind{\norm{\bsx_j}>s}$.
\begin{lemma}[Disjoint blocks II]\label{lem:covariance-disjoint-blocks-exc}
Assume that \ref{eq:conditiondh}, \ref{eq:rnbarFun0}, \eqref{eq:mixing-rates-1a} hold.
Then
\begin{align*}
&\lim_{n\to\infty}
\frac{1}{\dhinterseq\pr(\norm{\bsX_0}>\tepseq)}\sup_{\xi'>1}\sup_{H\in\widetilde\mca\cup\widetilde\mcb}
\sup_{s\in [s_0,t_0]}\esp\left[H(\bsX_{1,\dhinterseq}/\tepseq)
\exc_s(\bsX_{1+[\xi' \dhinterseq],\dhinterseq+[\xi'\dhinterseq]}/\tepseq)\right]=0\;.
\end{align*}
\end{lemma}
\begin{proof}
Recall that $\ell_n=o(\dhinterseq)$. Split the sum $\sum_{j=[\xi'\dhinterseq]+1}^{\dhinterseq+[\xi'\dhinterseq]}$ into
$\sum_{j=[\xi'\dhinterseq]+1}^{\dhinterseq+\ell_n}$  and $\sum_{j=\dhinterseq+\ell_n+1}^{\dhinterseq+[\xi'\dhinterseq]}$.

For the first sum we have
\begin{align*}
&\frac{1}{\dhinterseq\pr(\norm{\bsX_0}>\tepseq)}\sum_{j=[\xi'\dhinterseq]+1}^{\dhinterseq+\ell_n}\left|\esp[H(\bsX_{1,\dhinterseq}/\tepseq)\ind{\norm{\bsX_j}>\tepseq s}]\right|\leq \|H\|_{\infty}\frac{\ell_n\pr(\norm{\bsX_0}>\tepseq s)}{\dhinterseq\pr(\norm{\bsX_0}>\tepseq)}=o(1)
\end{align*}
uniformly over the class of functions and over $s$.

Using \eqref{eq:davydov-2} we have
\begin{align*}
&\frac{1}{\dhinterseq\pr(\norm{\bsX_0}>\tepseq)}\sum_{j=\dhinterseq+\ell_n+1}^{\dhinterseq+[\xi'\dhinterseq]}
\left|\cov(H(\bsX_{1,\dhinterseq}/\tepseq),\ind{\norm{\bsX_j}>\tepseq s})\right|\\
&\leq
\frac{\|H\|_{\infty}}{\dhinterseq\pr(\norm{\bsX_0}>\tepseq)}\sum_{j=\dhinterseq+\ell_n+1}^{\infty}\beta_{j-\dhinterseq}\;.
\end{align*}
We finish the proof using the mixing assumption \eqref{eq:mixing-rates-1a}.
\end{proof}
\subsubsection{Covariance of clusters: Overlapping blocks}
We consider three cases separately: a) $H,\widetilde{H}\in\mca$ (\Cref{thm:sliding-block-cov}); b) $H,\widetilde{H}\in\mcb$ (\Cref{thm:sliding-block-cov-hk}); c) the excess functional (\Cref{thm:sliding-block-cov-exc}).
\begin{proposition}[Overlapping blocks I]\label{thm:sliding-block-cov}
Assume that \ref{eq:conditiondh} and \ref{eq:rnbarFun0} hold. Let $h<\dhinterseq$ and $\xi\in (0,1)$.
For $H,\widetilde{H}\in\mca$ we have
\begin{align}\label{eq:covariance-sliding-blocks}
&\lim_{n\to\infty}
\frac{1}{\dhinterseq\pr(\norm{\bsX_0}>\tepseq)}\esp\left[H(\bsX_{1,\dhinterseq}/\tepseq)
\widetilde{H}(\bsX_{1+h,\dhinterseq+h}/\tepseq)\right]\notag\\
&=
\left\{
\begin{array}{ll}
\esp[H(\bsY)\widetilde{H}(\bsY)\ind{\bsY^*_{-\infty,-1}\leq 1}] \;, & \mbox{\rm if } h/\dhinterseq \to 0\;, \\
(1-\xi)\esp[H(\bsY)\widetilde{H}(\bsY)\ind{\bsY^*_{-\infty,-1}\leq 1}]\;, & \mbox{\rm if } h=h_n=[\xi \dhinterseq]
\end{array}
\right..
\end{align}
\end{proposition}
\begin{proof}[Proof of \Cref{thm:sliding-block-cov}]
Note that if $H,\widetilde{H}\in \mca$, then $H\widetilde{H}\in\mca$ (but it does not mean that we can apply \Cref{lem:cond-conv-H} since here the functions are applied to different blocks).

Since $H$ vanishes around $\bszero$, there exists $\epsilon>0$ such that $H(\bsx_{1,\dhinterseq})=0$ whenever $\bsx_{1,\dhinterseq}^*<\epsilon$. Assume \withoutlog\ that $\epsilon=1$. Then, splitting the event
$\{\bsX_{1,\dhinterseq}^*>\tepseq\}$ and using stationarity we write the expression of interest as
\begin{align*}
&\frac{1}{\dhinterseq\pr(\norm{\bsX_0}>\tepseq)}\esp\left[H(\bsX_{1,\dhinterseq}/\tepseq)
\widetilde{H}(\bsX_{1+h,\dhinterseq+h}/\tepseq)\ind{\bsX_{1,\dhinterseq}^*>\tepseq}\right]\\
&=\frac{1}{\dhinterseq\pr(\norm{\bsX_0}>\tepseq)}\sum_{j=1}^{\dhinterseq}
\esp\left[H(\bsX_{1,\dhinterseq}/\tepseq)\widetilde{H}(\bsX_{1+h,\dhinterseq+h}/\tepseq)\ind{\bsX^*_{1,j-1}\leq \tepseq}\ind{\norm{\bsX_j}>\tepseq}\right]\\
&=\frac{1}{\dhinterseq}
\sum_{j=1}^{\dhinterseq}\esp\left[H(\bsX_{1-j,\dhinterseq-j}/\tepseq)\widetilde{H}(\bsX_{1+h-j,\dhinterseq+h-j}/\tepseq)
\ind{\bsX^*_{1-j,-1}\leq \tepseq}\mid\norm{\bsX_0}>\tepseq\right]\;.
\end{align*}
We write the last expression as
$\int_0^1 g_n(v) \rmd v$
  with
  \begin{align*}
    g_n(v) = \esp[H( \bsX_{1-[\dhinterseq v],\dhinterseq-[\dhinterseq v]}/\tepseq)
    \widetilde{H}(\bsX_{1+h-[\dhinterseq v],\dhinterseq+h-[\dhinterseq v]}/\tepseq)
      \ind{\bsX_{1-[\dhinterseq v],-1}^*\leq \tepseq} \mid \norm{\bsX_0}>\tepseq ]
       \;.
  \end{align*}
If $h=o(\dhinterseq)$, then using the second part of \Cref{lem:cond-conv-H} we get $$\lim_{n\to\infty}g_n(v)=\esp[H(\bsY)\widetilde{H}(\bsY)\ind{\bsY^*_{-\infty,-1}\leq 1}]$$ independently of $v\in (0,1)$.
If $h=[\xi \dhinterseq]$, $\xi\in (0,1)$, then we split the integral.

If $\xi>v$, then we use boundedness of both $H,\widetilde{H}$ and the fact that $\widetilde{H}$ vanishes around $\bszero$. Thanks to the anticlustering condition \ref{eq:conditiondh}, we have as $n\to\infty$,
\begin{align*}
&\esp[H( \bsX_{1-[\dhinterseq v],\dhinterseq-[\dhinterseq v]}/\tepseq)
    \widetilde{H}(\bsX_{1+h-[\dhinterseq v],\dhinterseq+h-[\dhinterseq v]}/\tepseq)
      \ind{\bsX_{1-[\dhinterseq v],-1}^*\leq \tepseq} \mid \norm{\bsX_0}>\tepseq ]\\
      &\leq \constant \pr\left(\bsX^*_{1+[\xi\dhinterseq]-[\dhinterseq v],\dhinterseq +[\xi\dhinterseq]-[\dhinterseq v]}>\tepseq\mid \norm{\bsX_0}>\tepseq \right)\\
      &\leq \constant \pr\left(\bsX^*_{[\dhinterseq(\xi-v)],3\dhinterseq}>\tepseq\mid \norm{\bsX_0}>\tepseq \right)\to 0\;.
\end{align*}
If $\xi\leq v$, then we apply the second part of \Cref{lem:cond-conv-H}:
\begin{align*}
&\lim_{n\to\infty}\esp[H( \bsX_{1-[\dhinterseq v],\dhinterseq-[\dhinterseq v]}/\tepseq)
    \widetilde{H}(\bsX_{1+h-[\dhinterseq v],\dhinterseq+h-[\dhinterseq v]}/\tepseq)
      \ind{\bsX_{1-[\dhinterseq v],-1}^*\leq \tepseq} \mid \norm{\bsX_0}>\tepseq ]\\
      &= \esp[H(\bsY)\widetilde{H}(\bsY)\ind{\bsY^*_{-\infty,-1}\leq 1}]\;.
\end{align*}
Since the sequence $\{g_n\}$ is uniformly bounded,
we have
\begin{align*}
\lim_{n\to\infty}\int_0^{\xi} g_n(v) \rmd v+ \lim_{n\to\infty}\int_{\xi}^1 g_n(v) \rmd v=0+(1-\xi)\esp[H(\bsY)\widetilde{H}(\bsY)\ind{\bsY^*_{-\infty,-1}\leq 1}]\;.
\end{align*}
\end{proof}
\begin{proposition}[Overlapping blocks II]\label{thm:sliding-block-cov-hk}
Assume that \ref{eq:conditiondh}, \ref{eq:AN-small} and \ref{eq:rnbarFun0} hold. Let $h<\dhinterseq$ and $\xi\in (0,1)$.
Then \eqref{eq:covariance-sliding-blocks} holds
for $H,\widetilde{H}\in\mcb$.
\end{proposition}
\begin{proof}[Proof of \Cref{thm:sliding-block-cov-hk}]
We mimic the proof of \Cref{theo:limit-lone-Q} (refer to that proof for the notation).
  Set $K^{\epsilon} = K\circ
  T^{\epsilon}$,
  $\widetilde{K}^{\epsilon} = \widetilde{K}\circ
  T^{\epsilon}$. Note that $H^{\epsilon}:=\ind{K^{\epsilon}>1}\in \mca$, $\widetilde{H}^{\epsilon}:=\ind{\widetilde{K}^{\epsilon}>1}\in \mca$ and hence
  $H^{\epsilon}\widetilde{H}^{\epsilon}\in \mca$; see the comment at the beginning of the proof of \Cref{thm:sliding-block-cov}.

  Fix $\eta\in(0,1)$ and $\zeta>0$. Let $L_K, L_{\widetilde{K}}$ be as in~(\ref{eq:hypo-K}) and choose $\epsilon>0$ such
  that
  \begin{align*}
    \limsup_{n\to\infty} \frac{2\pr(\sum_{j=1}^{\dhinterseq} \norm{\bsX_{j}}\ind{\norm{\bsX_j}\leq\epsilon
        \tepseq} > \eta \tepseq/(L_K\vee L_{\widetilde{K}}))} {\dhinterseq\pr(\norm{\bsX_0}>\tepseq)} \leq \zeta \; .
  \end{align*}
This is allowed thanks to \ref{eq:AN-small}.
We have
\begin{align*}
&\frac{1}{\dhinterseq\pr(\norm{\bsX_0}>\tepseq)}\pr\left(K(\bsX_{1,\dhinterseq}/\tepseq)>1,
\widetilde{K}(\bsX_{1+h,\dhinterseq+h}/\tepseq)>1\right)\\
     & \leq
     \frac{\pr(K^{\epsilon}(\bsX_{1,\dhinterseq}/\tepseq) > 1-\eta, \widetilde{K}^{\epsilon}(\bsX_{1+h,\dhinterseq+h}/\tepseq)>1-\eta)} {\dhinterseq\pr(\norm{\bsX_0}>\tepseq)}\\
    &\phantom{=} +
      \frac{2\pr(\sum_{i=1}^{\dhinterseq} \norm{\bsX_{j}}\ind{\norm{\bsX_j}\leq\epsilon \tepseq} > \eta \tepseq/\constant)}
      {\dhinterseq\pr(\norm{\bsX_0}>\tepseq)}
       \; .
\end{align*}
Application of \Cref{thm:sliding-block-cov} gives
\begin{align*}
&\lim_{n\to\infty}\frac{1}{\dhinterseq\pr(\norm{\bsX_0}>\tepseq)}\pr\left(K(\bsX_{1,\dhinterseq}/\tepseq)>1,
\widetilde{K}(\bsX_{1+h,\dhinterseq+h}/\tepseq)>1\right)\\
&\leq\zeta+\left\{
\begin{array}{ll}
\esp[H^{\epsilon}(\bsY)\widetilde{H}^{\epsilon}(\bsY)\ind{\bsY^*_{-\infty,-1}\leq 1}] \;, & \mbox{\rm if } h/\dhinterseq \to 0\;, \\
(1-\xi)\esp[H^{\epsilon}(\bsY)\widetilde{H}^{\epsilon}(\bsY)\ind{\bsY^*_{-\infty,-1}\leq 1}]\;, & \mbox{\rm if } h=h_n=[\xi \dhinterseq]
\end{array}
\right..\end{align*}
Similarly, we obtain the lower bound with $1+\eta$ instead of $1-\eta$ and $-\zeta$ instead of $+\zeta$.
Since $\zeta$ is arbitrary, the proof is concluded by letting $\epsilon\to 0$. This follows the same argument as in the proof of \Cref{theo:limit-lone-Q}.
\end{proof}
\begin{proposition}[Overlapping blocks III]\label{thm:sliding-block-cov-exc}
Assume that \ref{eq:conditiondh} and \ref{eq:rnbarFun0} hold. Let $h<\dhinterseq$.
For $H\in\mcl$ we have
\begin{align*}
&\lim_{n\to\infty}
\frac{1}{\dhinterseq\pr(\norm{\bsX_0}>\tepseq)}\esp\left[H(\bsX_{1,\dhinterseq}/\tepseq)
\exc_s(\bsX_{1+h,\dhinterseq+h}/\tepseq)\right]\\
&=
\left\{
\begin{array}{ll}
s^{-\alpha}\esp[H(s\bsY)] \;, & \mbox{\rm if } h/\dhinterseq \to 0\;, \\
s^{-\alpha}(1-\xi)\esp[H(s\bsY)]\;, & \mbox{\rm if } h=h_n=\xi \dhinterseq
\end{array}
\right..
\end{align*}
\end{proposition}
\begin{proof}
[Proof of \Cref{thm:sliding-block-cov-exc}]
We have for $h<\dhinterseq$,
\begin{align*}
&\frac{1}{\dhinterseq\pr(\norm{\bsX_0}>\tepseq)}
\esp\left[H(\bsX_{1,\dhinterseq}/\tepseq)\exc_s(\bsX_{1+h,\dhinterseq+h}/\tepseq)\right]\\
&=\frac{1}{\dhinterseq\pr(\norm{\bsX_0}>\tepseq)}\sum_{j=h+1}^{\dhinterseq+h}\esp[H(\bsX_{1,\dhinterseq}/\tepseq)\ind{\norm{\bsX_j}>\tepseq s}]\\
&= \frac{1}{\dhinterseq}\frac{\pr(\norm{\bsX_0}>\tepseq s)}{\pr(\norm{\bsX_0}>\tepseq )}\sum_{j=h+1}^{\dhinterseq+h}
\esp[H(\bsX_{1-j,\dhinterseq-j}/\tepseq)\mid \norm{\bsX_0}>\tepseq s]\;.
\end{align*}
We write the last expression as
\begin{align*}
&\frac{\pr(\norm{\bsX_0}>\tepseq s)}{\pr(\norm{\bsX_0}>\tepseq )}\int_{h/\dhinterseq}^{1+h/\dhinterseq}g_n(v)\rmd v
\end{align*}
with (omitting the dependence on $s$)
\begin{align*}
g_n(v)=\esp[H(s\bsX_{1-[\dhinterseq v],\dhinterseq-[\dhinterseq v]}/(\tepseq s))\mid \norm{\bsX_0}>\tepseq s]\;.
\end{align*}
Since $H$ is bounded, \ref{eq:conditiondh} and \Cref{lem:tailprocesstozero} give
\begin{align*}
\lim_{n\to\infty} g_n(v) =\left\{
\begin{array}{ll}
\esp[H(s\bsY)] & \mbox{\rm if } v\in (0,1)\;, \\
0 & \mbox{\rm if } v>1\;.
\end{array}
\right.
\end{align*}
We split
\begin{align*}
\int_{h/\dhinterseq}^{1+h/\dhinterseq}g_n(v)\rmd v=\int_{h/\dhinterseq}^{1}g_n(v)\rmd v+\int_{1}^{1+h/\dhinterseq}g_n(v)\rmd v\;.
\end{align*}
Since the sequence $\{g_n\}$ is uniformly bounded, for any $h<\dhinterseq$ the second integral above converges to zero as $n\to\infty$. If $h=o(\dhinterseq)$ and since there is no problem at $v=0$ with $g_n(v)$, then the first integral converges to
\begin{align*}
\int_0^1 \esp[H(s\bsY)] \rmd v =\esp[H(s\bsY)]\;.
\end{align*}
Likewise, when $h=[\xi\dhinterseq]$ then the first integral converges to
\begin{align*}
\int_{\xi}^1 \esp[H(s\bsY)] \rmd v =(1-\xi)\esp[H(s\bsY)]\;.
\end{align*}
\end{proof}

\subsection{Empirical cluster process of sliding blocks}\label{sec:empirical-process}
Recall that for $s>0$, $H_s(\bsx)=H(\bsx/s)$.
In order to deal with asymptotic normality of sliding blocks  estimators, we study the empirical process
\begin{align*}
\mathbb{F}_n(H_s)&:=\sqrt{\statinterseqn}
\left\{
\tedclustersl(H_s)-\tailmeasurestar(H_s)\right\}
=\sqrt{\statinterseqn}
\left\{
\frac{\sum_{i=0}^{q_n-1}H_s\left(\bsX_{i+1,i+\dhinterseq}/\tepseq\right)}{q_n \dhinterseq \pr(\norm{\bsX_0}>\tepseq)}-s^{-\alpha}\tailmeasurestar(H)\right\}\;.
\end{align*}
The process $\mathbb{F}_n(H_s)$ is viewed as a random element with values in
$\spaceD([s_0,t_0])$.
\begin{theorem}
  \label{thm:sliding-blocks-process-clt}
  Let $\sequence{\bsX}$ be a stationary, regularly varying $\Rset^d$-valued time series.
  Assume
  that~\ref{eq:rnbarFun0}, $\beta(\dhinterseq)$ and
  \ref{eq:conditiondh} hold.
  Let $H\in \mca$ be such that the class $\{H_s:s\in [s_0,t_0]\}$ is linearly ordered and \eqref{eq:bias-cluster-clt-func}
  holds.

  Then ${\mathbb{F}}_n(H_{\cdot})$ converges weakly in $(\spaceD([s_0,t_0]),J_1)$ to a Gaussian process with the covariance
 $\tailmeasurestar(H_sH_t)$.

If moreover \ref{eq:AN-small} is satisfied, then the convergence holds for $H\in\mcb$.

 If additionally \ref{eq:conditionS} and \eqref{eq:bias-cluster-clt-aa} are satisfied, then
 the processes ${\mathbb{F}}_n(H_{\cdot})$ and ${\mathbb{F}}_n(\exc_{\cdot})$
 converge jointly.
\end{theorem}
\subsubsection{Tail empirical process}\label{sec:tep}
Consider the following tail empirical process:
\begin{align*}
  \TEPunivsl_n(s) &  =\sqrt{\statinterseqn} \left\{T_n(s)-s^{-\alpha}\right\}=\sqrt{\statinterseqn} \left\{\frac{
  \sum_{j=1}^{q_n}\ind{\norm{\bsX_j}>\tepseq s}}{q_n\pr(\norm{\bsX_0}>\tepseq)} - s^{-\alpha}\right\} \; , \ s > 0\; .
\end{align*}
Note that this is the classical tail empirical process based on the random variables $\norm{\bsX_j}$, $j\geq 1$, with the only one difference: $q_n$ replaces $n$. We argue that this process can be obtained (approximately) as the empirical process of sliding blocks.
Indeed,
\begin{align*}
\tedclustersl(\exc_s)
&=\frac{1}{q_n\dhinterseq \pr(\norm{\bsX_0}>\tepseq)}
\left\{\sum_{j=1}^{\dhinterseq}j+\dhinterseq\sum_{j=\dhinterseq+1}^{q_n}+\sum_{j=q_n+1}^n(n-j)\right\}
\ind{\norm{\bsX_j}> \tepseq s}\;.
\end{align*}
The difference between $\tedclustersl(\exc_s)$ and $T_n(s)$ is
\begin{align*}
A:=\frac{1}{q_n\dhinterseq \pr(\norm{\bsX_0}>\tepseq)}\left\{\sum_{j=1}^{\dhinterseq}(\dhinterseq-j)-
\sum_{j=q_n+1}^n(n-j)\right\}
\ind{\norm{\bsX_j}> \tepseq s}\;.
\end{align*}
We have $\sum_{j=1}^{\dhinterseq}(\dhinterseq-j)\leq \dhinterseq^2$ and $\sum_{j=q_n+1}^n(n-j)\leq \dhinterseq^2$, thus
under \ref{eq:rnbarFun0}:
$$
\lim_{n\to\infty}\sqrt{\statinterseqn}\esp[|A|]\leq \constant \lim_{n\to\infty}\sqrt{n\pr(\norm{\bsX_0}>\tepseq)}\frac{\dhinterseq}{q_n}=\constant
\lim_{n\to\infty} \sqrt{\frac{\dhinterseq}{n}}\sqrt{\dhinterseq \pr(\norm{\bsX_0}>\tepseq)}=0\;.
$$
This implies that ${\mathbb{F}}_n(\exc_s)$ and $\TEPunivsl_n(s)$ are asymptotically equivalent in the sense that they yield the same process ${\mathbb{F}}(\exc_s)$ as the distributional limit.
\subsection{Covariance of the empirical process of sliding blocks}
\label{sec:covariance-empirical-process}
\begin{proposition}\label{prop:covariance-Gn}
Assume that \ref{eq:conditiondh} and \ref{eq:rnbarFun0} are satisfied. Let
\begin{itemize}
\item
$H,\widetilde{H}\in\widetilde{\mca}$, or
\item
$H,\widetilde{H}\in\widetilde{\mcb}$
and \ref{eq:AN-small} holds.
\end{itemize}
If \eqref{eq:mixing-rates-1a} and \eqref{eq:mixing-rates-7} hold then
\begin{align}\label{eq:limiting-cov-H}
\lim_{n\to\infty}\cov(\mathbb{F}_n(H),\mathbb{F}_n(\widetilde{H}))=  \tailmeasurestar(H\widetilde{H})\;.
\end{align}
If \eqref{eq:mixing-rates-1a} holds then
\begin{align}\label{eq:limiting-cov-H-exc}
\lim_{n\to\infty}\cov(\mathbb{F}_n(H),\mathbb{F}_n(\exc_s))= \tailmeasurestar(H\exc_s)=\esp[H(s\bsY)]\;.
\end{align}
\end{proposition}
\begin{remark}
\begin{itemize}
\item
The second equality in \eqref{eq:limiting-cov-H-exc} follows from \Cref{exer:expression-tailmeasurestar-HH'}.
\item In view of the discussion in \Cref{sec:tep}, \eqref{eq:limiting-cov-H-exc} can be re-phrased as
\begin{align*}
\lim_{n\to\infty}\cov(\mathbb{F}_n(H),\TEPuniv_n(s))= \tailmeasurestar(H\exc_s)=\esp[H(s\bsY)]\;.
\end{align*}
\end{itemize}
\remarkend
\end{remark}
\subsubsection{Bounds for integral representation}
Before we proceed with the proof, we define
$$
g_n(\xi;H)=\esp\left[H(\bsX_{1,\dhinterseq}/\tepseq)H(\bsX_{1+[\dhinterseq \xi],[\dhinterseq \xi]+\dhinterseq}/\tepseq)\right]\;, \ \ \xi>0\;
$$
and
$$
\widetilde g_n(\xi;H)=\frac{g_n(\xi;H)} {\dhinterseq \pr(\norm{\bsX_0}>\tepseq)}\;.
$$
For $\xi=0$, using \Cref{rem:A-LinA-delta} we immediately obtain under \ref{eq:conditiondh}:
\begin{align}\label{eq:gntilde-0}
\lim_{n\to\infty}\sup_{H\in \widetilde\mca\cup\widetilde\mcb}\widetilde{g}_n(0;H)
=\lim_{n\to\infty}\sup_{H\in \widetilde\mca\cup\widetilde\mcb}\tailmeasurestar_{n,\dhinterseq}(H^2)<\infty\;.
\end{align}
Furthermore,  for $j=1,2,3,\ldots$,
$$
\frac{1}{\dhinterseq}
\sum_{i=(j-1)\dhinterseq}^{j\dhinterseq-1}\widetilde{g}_{n}(i/\dhinterseq;H)=\int_{j-1}^j\widetilde{g}_n(\xi;H)\rmd \xi\;.
$$
For $j=1$ we will need the precise behaviour of this integral and we will
handle it using
\Cref{thm:sliding-block-cov,thm:sliding-block-cov-hk}. For $j\geq 2$ the integral vanishes with a given rate.
\begin{lemma}\label{lem:integral-negligible}
Assume that \ref{eq:conditiondh} holds.
\begin{itemize}
\item
If \eqref{eq:mixing-rates-1a} holds then for any finite $M$,
\begin{align*}
\lim_{n\to\infty}\sup_{H\in \widetilde\mca\cup\widetilde\mcb}\int_{1}^M\widetilde{g}_n(\xi;H)\rmd \xi=0
\end{align*}
\item For $j\geq 3$,
\begin{align*}
\sup_{H\in \widetilde\mca\cup\widetilde\mcb}\int_{j-1}^j\widetilde{g}_n(\xi;H)\rmd \xi\leq \constant\
\frac{1}{\dhinterseq \pr(\norm{\bsX_0}>\tepseq)}\beta_{(j-2)\dhinterseq}\;.
\end{align*}
\end{itemize}
\end{lemma}
\begin{proof}
For the first part we apply \Cref{lem:covariance-disjoint-blocks} and the dominated convergence:
$$\sup_{H\in\widetilde\mca\cup\widetilde\mcb}\sup_{\xi\in (1,2)}|\widetilde g_n(\xi;H)|
\leq
\sup_{H\in\widetilde\mca\cup\widetilde\mcb}\|H\|_{\infty}
\tailmeasurestar_{n,\dhinterseq}(|H|)\leq \constant\ \sup_{H\in\widetilde\mca\cup\widetilde\mcb}\tailmeasurestar_{n,\dhinterseq}(|H|)<\infty \;. $$
For the second part, we use \eqref{eq:davydov-2} and the fact that $\widetilde\mca\cup\widetilde\mcb$ has a finite envelope.
\end{proof}
\subsubsection{Representation for covariance between blocks}
Recall that $q_n=n-\dhinterseq+1$. Evaluation of the covariance of the empirical process of sliding blocks will use consecutive disjoint blocks of indices of size $\dhinterseq$:
$$
J_j=\{(j-1)\dhinterseq,\ldots,j\dhinterseq-1\}\;, \ \ j=1,\ldots,m_n=[q_n/\dhinterseq]\;.
$$
Clearly, $\bigcup_{j=1}^{m_n}J_j=\left\{0,\ldots,n-\dhinterseq\right\}$.
We will assume for simplicity that $q_n/\dhinterseq$ is an integer.

Write
\begin{align*}
\frac{1}{q_n \dhinterseq \pr(\norm{\bsX_0}>\tepseq)}\sum_{i=0}^{q_n-1}H\left(\bsX_{i+1,i+\dhinterseq}/\tepseq\right)
=\frac{1}{q_n \dhinterseq \pr(\norm{\bsX_0}>\tepseq)}\sum_{j=1}^{m_n}\Psi_j(H)
\end{align*}
with
\begin{align}\label{eq:block-of-size-rn}
\Psi_j(H)=
\sum_{i\in J_j}H\left(\bsX_{i+1,i+\dhinterseq}/\tepseq\right)\;.
\end{align}
Note that the indices of the random vectors
$\bsX_{1},\ldots,\bsX_{2\dhinterseq-1}$ used in the construction of $\Psi_1$ overlap with the indices of
$\bsX_{\dhinterseq+1},\ldots,\bsX_{3\dhinterseq-1}$ used to define $\Psi_2$, but do not overlap with the indices used in the definition of $\Psi_3$. Likewise, the indices used in the definition of $\Psi_2$ overlap with those in $\Psi_3$, but not with any other term $\Psi_j$, $j\geq 4$. This partially explains where does a contribution to the limiting variance come from: from the dependence within each block $J_j$ and cross dependence between $J_j$ and two neighbouring blocks.

For $j\geq 1$ we have
\begin{align*}
&\frac{\esp[\Psi_1(H)\Psi_{j+1}(H)]}{\dhinterseq^3 \pr(\norm{\bsX_0}>\tepseq)}=
\frac{1}{\dhinterseq^3 \pr(\norm{\bsX_0}>\tepseq)}
\esp\left[\sum_{h=0}^{\dhinterseq-1}H\left(\bsX_{h+1,h+\dhinterseq}/\tepseq\right)
\sum_{i=j\dhinterseq}^{(j+1)\dhinterseq-1}H\left(\bsX_{i+1,i+\dhinterseq}/\tepseq\right)\right]\\
&=\frac{1}{\dhinterseq}
\sum_{i=(j-1)\dhinterseq}^{j\dhinterseq}\left(\frac{i}{\dhinterseq}-(j-1)\right)\widetilde{g}_{n}(i/\dhinterseq;H)
+\frac{1}{\dhinterseq}
\sum_{i=j\dhinterseq+1}^{(j+1)\dhinterseq}\left((j+1)-\frac{i}{\dhinterseq}\right)\widetilde{g}_{n}(i/\dhinterseq;H)
\end{align*}
and hence
\begin{align}
&\frac{\esp[|\Psi_1(H)\Psi_{j+1}(H)|]}{\dhinterseq^3 \pr(\norm{\bsX_0}>\tepseq)}
\leq \frac{1}{\dhinterseq}
\sum_{i=(j-1)\dhinterseq}^{(j+1)\dhinterseq}|\widetilde{g}_{n}(i/\dhinterseq;H)|
\leq
\int_{j-1}^{j+1}|\widetilde{g}_n(\xi;H)|\rmd \xi\;. \label{eq:covariance-bound}
\end{align}

\subsubsection{Proof of \Cref{prop:covariance-Gn}, Eq. \eqref{eq:limiting-cov-H}}
\begin{proof}
Note that (since $q_n\sim n$)
\begin{align}\label{eq:ratio}
\frac{\statinterseqn m_n}{\left(q_n \dhinterseq \pr(\norm{\bsX_0}>\tepseq)\right)^2}=
\frac{q_nn\pr(\norm{\bsX_0}>\tepseq)}{\dhinterseq \left(q_n \dhinterseq \pr(\norm{\bsX_0}>\tepseq)\right)^2}
\sim\frac{1}{\dhinterseq^3 \pr(\norm{\bsX_0}>\tepseq)}\;.
\end{align}
Write $\var(\mathbb{F}_n(H))$ as
\begin{align}\label{eq:ourdecomp}
\frac{\statinterseqn m_n}{\left(q_n \dhinterseq \pr(\norm{\bsX_0}>\tepseq)\right)^2}
\cov(\Psi_2(H),\Psi_1(H)+\Psi_2(H)+\Psi_3(H))+ A_n(H)
\end{align}
with the reminder $A_n(H)$ given by
\begin{align}
A_n(H):=&-2 \frac{\statinterseqn}{\left(q_n \dhinterseq \pr(\norm{\bsX_0}>\tepseq)\right)^2}
\cov(\Psi_1(H),\Psi_2(H))\label{eq:ourdecomp1} \\
&+2 \frac{1+o(1)}{\dhinterseq^3 \pr(\norm{\bsX_0}>\tepseq)}  \sum_{j=2}^{m_{n}-1}\left(1-\frac{j}{m_{n}}\right)\left\{\cov\left(\Psi_1(H), \Psi_{1+j}(H)\right)\right\}\label{eq:ourdecomp2} \\
&=:A_{n,1}(H)+(1+o(1))B_n(H)\;. \notag
\end{align}
If we show that the leading term on the \rhs\ of \eqref{eq:ourdecomp} converges to a finite limit,
then automatically $\lim_{n\to\infty}A_{n,1}(H)=0$ (since $m_n\to\infty)$. Thus, the reminder $A_n(H)$ will be negligible if we show that
\begin{align}\label{eq:reminder-1}
\lim_{n\to\infty}B_n(H)=0\;.
\end{align}
We will start by analysing the first term in \eqref{eq:ourdecomp}.
Set
$$
R_n(H)=\frac{1}{\dhinterseq}\sum_{i=\dhinterseq}^{2\dhinterseq-1}(2-i/\dhinterseq)\widetilde{g}_{n}(i/\dhinterseq;H)\;.
$$
Since \eqref{eq:mixing-rates-1a} holds, the application of the first part of \Cref{lem:integral-negligible} gives
\begin{align}\label{eq:reminder-2-mcg}
\lim_{n\to\infty}\sup_{H\in \widetilde\mca\cup\widetilde\mcb}R_n(H)=0\;.
\end{align}
Write the first term in \eqref{eq:ourdecomp} as (cf. \eqref{eq:ratio})
\begin{align}\label{eq:variance-decomposition-leading-term}
&
\frac{1+o(1)}{\dhinterseq^3 \pr(\norm{\bsX_0}>\tepseq)}
\textcolor{black}{\left\{
\dhinterseq g_{n}(0;H)+2\dhinterseq \sum_{i=1}^{\dhinterseq-1}g_{n}(i/\dhinterseq;H)+2\dhinterseq R_n\right\}}\notag\\
&=(1+o(1))\left\{\frac{1}{\dhinterseq}\widetilde{g}_{n}(0;H)+
2\frac{1}{\dhinterseq}
\sum_{i=1}^{\dhinterseq-1}\widetilde{g}_{n}(i/\dhinterseq;H)+
2R_n(H)\right\}\;.
\end{align}
Then, using \eqref{eq:gntilde-0}, \eqref{eq:reminder-1} and \eqref{eq:reminder-2-mcg}, we have
\begin{align*}
\lim_{n\to\infty}\var(\mathbb{F}_n(H))
&=2\lim_{n\to\infty}\int_0^1
\widetilde g_n(\xi;H)\rmd \xi\;,
\end{align*}

Applying \Cref{thm:sliding-block-cov,thm:sliding-block-cov-hk} (the case $h=[\xi \dhinterseq]$), we have
\begin{align*}
\lim_{n\to\infty}\var(\mathbb{F}_n(H))= 2 \tailmeasurestar(H^2)\int_0^1 (1-\xi)\rmd \xi= \tailmeasurestar(H^2)\;.
\end{align*}
To conclude the proof, we show \eqref{eq:reminder-1} in the following lemma.
\end{proof}
\begin{lemma}\label{lem:term-Bn}
Assume that \eqref{eq:mixing-rates-1a}-\eqref{eq:mixing-rates-7}
hold. Then
\begin{align}\label{eq:reminder-1-mcg}
\lim_{n\to\infty}\sup_{H\in \widetilde\mca\cup\widetilde\mcb}B_n(H)=0\;.
\end{align}
\end{lemma}
\begin{proof}[Proof of \Cref{lem:term-Bn}]
Using \eqref{eq:covariance-bound} we have
\begin{align*}
&|B_n(H)|
\leq \frac{\esp[|\Psi_1(H)\Psi_3(H)|]}{\dhinterseq^3 \pr(\norm{\bsX_0}>\tepseq)}
+\frac{1}{\dhinterseq^3 \pr(\norm{\bsX_0}>\tepseq)}
\sum_{j=3}^{m_{n}-1}|\esp[|\Psi_1(H)\Psi_{1+j}(H)|]\\
&\leq\int_{1}^{3}|\widetilde{g}_n(\xi;H)|\rmd \xi+
\sum_{j=3}^{m_{n}-1}\int_{j-1}^{j+1}|\widetilde{g}_n(\xi;H)|\rmd \xi\;.
\end{align*}
The first term is $o(1)$ uniformly over the class of functions (cf. the first part of
\Cref{lem:integral-negligible}). Using the second part of \Cref{lem:integral-negligible} we bound
$$
\sum_{j=3}^{m_{n}-1}\int_{j-1}^{j+1}|\widetilde{g}_n(\xi;H)|\rmd \xi\leq\constant \frac{1}{\dhinterseq \pr(\norm{\bsX_0}>\tepseq)}
\sum_{j=1}^\infty\beta_{j\dhinterseq}\;.
$$
We finish the proof by applying the mixing assumption \eqref{eq:mixing-rates-7}.
\end{proof}
\subsubsection{Proof of \Cref{prop:covariance-Gn}, Eq. \eqref{eq:limiting-cov-H-exc}}
\begin{proof}
We write (recall that $q_n\sim n$)
\begin{align*}
&\statinterseqn\cov\left(\frac{1}{q_n \dhinterseq \pr(\norm{\bsX_0}>\tepseq)}\sum_{i=0}^{q_n-1}H(\bsX_{i+1,i+\dhinterseq}/\tepseq),\frac{1}{q_n\pr(\norm{\bsX_0}>\tepseq)}\sum_{j=1}^{q_n}\ind{\norm{\bsX_j}>\tepseq}\right)\\
&\sim \frac{\statinterseqn}{n^2\dhinterseq \pr^2(\norm{\bsX_0}>\tepseq)}
\sum_{i=0}^{q_n-1}\sum_{j=1}^{q_n}\cov\left(H(\bsX_{i+1,i+\dhinterseq}/\tepseq),\ind{\norm{\bsX_j}>\tepseq}\right)\\
&=\frac{1}{n\dhinterseq \pr(\norm{\bsX_0}>\tepseq)}
\sum_{i=0}^{q_n-1}\sum_{j=1}^{q_n}\cov\left(H(\bsX_{i-j+1,i-j+\dhinterseq}/\tepseq),\ind{\norm{\bsX_0}>\tepseq}\right)\;.
\end{align*}

Split the inner sum into two pieces, $\sum_{j=1}^{i}$ and $\sum_{j=i+1}^{q_n}$, in the first one replace $j$ with $h=i-j$, in the second one replace $j$ with $h=j-i$ to get
\begin{align*}
&\frac{1}{q_n\dhinterseq \pr(\norm{\bsX_0}>\tepseq)}
\sum_{i=0}^{q_n-1}\sum_{j=1}^{i}\cov\left(H(\bsX_{i-j+1,i-j+\dhinterseq}/\tepseq),\ind{\norm{\bsX_0}>\tepseq}\right)\\
&\phantom{=}+\frac{1}{q_n\dhinterseq \pr(\norm{\bsX_0}>\tepseq)}
\sum_{i=0}^{q_n-1}\sum_{j=i+1}^{q_n}\cov\left(H(\bsX_{i-j+1,i-j+\dhinterseq}/\tepseq),\ind{\norm{\bsX_0}>\tepseq}\right)\\
&=\frac{1}{q_n\dhinterseq \pr(\norm{\bsX_0}>\tepseq)}
\sum_{i=1}^{q_n-1}\sum_{h=0}^{i-1}\cov\left(H(\bsX_{h+1,h+\dhinterseq}/\tepseq),\ind{\norm{\bsX_0}>\tepseq}\right)\\
&\phantom{=}+\frac{1}{q_n\dhinterseq \pr(\norm{\bsX_0}>\tepseq)}
\sum_{i=0}^{q_n-1}\sum_{h=1}^{q_n-i}\cov\left(H(\bsX_{-h+1,-h+\dhinterseq}/\tepseq),\ind{\norm{\bsX_0}>\tepseq}\right)\;.
\end{align*}
This gives further
\begin{align}
&\frac{1}{q_n\dhinterseq \pr(\norm{\bsX_0}>\tepseq)}
\sum_{h=0}^{q_n-2}\sum_{i=h+1}^{q_n-1}\cov\left(H(\bsX_{h+1,h+\dhinterseq}/\tepseq),\ind{\norm{\bsX_0}>\tepseq}\right)\notag\\
&\phantom{=}+\frac{1}{q_n\dhinterseq \pr(\norm{\bsX_0}>\tepseq)}
\sum_{h=1}^{q_n-2}\sum_{i=1}^{q_n-h}\cov\left(H(\bsX_{-h+1,-h+\dhinterseq}/\tepseq),\ind{\norm{\bsX_0}>\tepseq}\right)\notag\\
&=\frac{1}{\dhinterseq \pr(\norm{\bsX_0}>\tepseq)}
\sum_{h=0}^{q_n-1}(1-h/q_n)\cov\left(H(\bsX_{h+1,h+\dhinterseq}/\tepseq),\ind{\norm{\bsX_0}>\tepseq}\right)\label{eq:second-term-1}\\
&\phantom{=}+\frac{1}{\dhinterseq \pr(\norm{\bsX_0}>\tepseq)}
\sum_{h=0}^{q_n-1}(1-h/q_n)\cov\left(H(\bsX_{-h+1,-h+\dhinterseq}/\tepseq),\ind{\norm{\bsX_0}>\tepseq}\right)\;.\label{eq:second-term}
\end{align}
We show that the term in \eqref{eq:second-term-1} is negligible, while the one in \eqref{eq:second-term} yields the limit.
We split the term in \eqref{eq:second-term-1} into two pieces, according to $h\leq \dhinterseq$ and $h>\dhinterseq$. Then  the first part is bounded by
\begin{align*}
\frac{1}{\dhinterseq}\sum_{h=0}^{\dhinterseq}\esp\left[|H(\bsX_{h+1,h+\dhinterseq}/\tepseq)| \left|\right. \norm{\bsX_0}>\tepseq\right]=\int_0^1 g_n(v)\rmd v
\end{align*}
with
\begin{align*}
g_n(v)=\esp\left[|H(\bsX_{[\dhinterseq v]+1,[\dhinterseq v]+\dhinterseq}/\tepseq)| | \norm{\bsX_0}>\tepseq\right]\;.
\end{align*}
Under \ref{eq:conditiondh}, $g_n(v)\to 0$ (cf. the first part of \Cref{lem:cond-conv-H} with $H\equiv 1$ and $\widetilde{H}=|H|$).
Likewise, since $\dhinterseq/q_n\to 0$,
\begin{align*}
\frac{1}{\dhinterseq q_n}\sum_{h=0}^{\dhinterseq}h\esp\left[|H(\bsX_{h+1,h+\dhinterseq}/\tepseq)| | \norm{\bsX_0}>\tepseq\right]=q_n^{-1}\int_0^1 [\dhinterseq v]g_n(v)\rmd v \to 0\;.
\end{align*}
Furthermore, applying \eqref{eq:davydov-2},
\begin{align*}
&\frac{1}{\dhinterseq \pr(\norm{\bsX_0}>\tepseq)}
\sum_{h=\dhinterseq+1}^{q_n-1}(1-h/q_n)|\cov\left(H(\bsX_{h+1,h+\dhinterseq}/\tepseq),\ind{\norm{\bsX_0}>\tepseq}\right)|\\
&\leq
\frac{\|H\|_{\infty}}{\dhinterseq \pr(\norm{\bsX_0}>\tepseq)}
\sum_{h=1}^{n}\beta_{h+\dhinterseq}
\end{align*}
and the latter term vanishes by \eqref{eq:mixing-rates-1a}. In summary, \eqref{eq:second-term-1} is negligible.

For the term in \eqref{eq:second-term} we write (recall that we can replace $\cov$ with $\esp$ thanks to \ref{eq:rnbarFun0})
\begin{align*}
\frac{1}{\dhinterseq}\sum_{h=0}^{\dhinterseq}\esp\left[H(\bsX_{-h+1,-h+\dhinterseq}/\tepseq)\mid \norm{\bsX_0}>\tepseq\right]=\int_0^1 g_n(v)\rmd v
\end{align*}
with
\begin{align*}
g_n(v)=\esp\left[H(\bsX_{-[\dhinterseq v]+1,-[\dhinterseq v]+\dhinterseq}/\tepseq)\mid \norm{\bsX_0}>\tepseq\right]\;.
\end{align*}
By \Cref{lem:tailprocesstozero}, $g_n(v)\to \esp[H(\bsY)]$ for each $v$.
\end{proof}
\subsection{Proof of \Cref{thm:sliding-blocks-process-clt} - fidi convergence}\label{sec:fidi}
Recall that $q_n=n-\dhinterseq+1$ and recall the disjoint blocks of size $\dhinterseq$:
$$
J_j=\{(j-1)\dhinterseq,\ldots,j\dhinterseq-1\}\;, \ \ j=1,\ldots,m_n=[q_n/\dhinterseq]\;.
$$
These blocks were chosen to calculate the limiting covariance of the process $\mathbb{F}_n$.
However, they are not appropriate for a proof of the central limit theorem. We need to introduce a large-small blocks decomposition.
For this purpose let $z_n$ be a sequence of integers such that $z_n\to\infty$
and
\begin{align}\label{eq:zn}
\lim_{n\to\infty}\frac{z_n}{\sqrt{n\pr(\norm{\bsX_0}>\tepseq)}}=0\;.
\end{align}
Set
$$\widetilde{m}_n=\left[\frac{q_n}{(z_n+2)\dhinterseq}\right]$$
and assume for simplicity that $\widetilde{m}_n$ is an integer. Since $z_n\to\infty$, we have $\widetilde{m}_n=o(m_n)$.
For $j=1,\ldots,\widetilde{m}_n$ define now large and small blocks as follows:
\begin{align*}
&L_1=\{0,\ldots, z_n\dhinterseq-1\}\;,\ \  S_1=\{z_n\dhinterseq\ldots,z_n\dhinterseq+2\dhinterseq-1\}\;, \\
&L_2=\{z_n\dhinterseq+2\dhinterseq,\ldots, 2z_n\dhinterseq+2\dhinterseq-1\}\;,\ \  S_2=\{2z_n\dhinterseq+2\dhinterseq\ldots,2z_n\dhinterseq+4\dhinterseq-1\}\;, \\
&L_j=\{(j-1)z_n\dhinterseq+2(j-1)\dhinterseq,\ldots, jz_n\dhinterseq+2(j-1)\dhinterseq-1\}\;,\\ &S_j=\{jz_n\dhinterseq+2(j-1)\dhinterseq,\ldots,jz_n\dhinterseq+2j\dhinterseq-1\}\;.
\end{align*}
The block $L_1$ is obtained by merging $z_n$ consecutive blocks $J_1,\ldots,J_{z_n}$ of size $\dhinterseq$. Likewise, $S_1=J_{z_n+1}\cup J_{z_n+2}$.
Therefore, the large block of size $z_n\dhinterseq$ is followed by the small block of size $2\dhinterseq$, which in turn is followed by the large block of size $z_n\dhinterseq$ and so on. All together,
\begin{align*}
\bigcup_{j=1}^{\widetilde{m}_n}\left(L_j\cup S_j\right)=\{0,\ldots,n-\dhinterseq\}\;.
\end{align*}
Write
\begin{align}\label{eq:large-small-block-process}
&\sum_{i=0}^{q_n-1}H\left(\bsX_{i+1,i+\dhinterseq}/\tepseq\right)
=\sum_{j=1}^{\widetilde{m}_n}\Psi_j^{(l)}(H)
+\sum_{j=1}^{\widetilde{m}_n}\Psi_j^{(s)}(H)\;,
\end{align}
where now
\begin{align*}
\Psi_j^{(l)}(H)=
\sum_{i\in L_j}H\left(\bsX_{i+1,i+\dhinterseq}/\tepseq\right)\;, \ \
\Psi_j^{(s)}(H)=
\sum_{i\in S_j}H\left(\bsX_{i+1,i+\dhinterseq}/\tepseq\right)
\;.
\end{align*}
With such the decomposition, $\bsX_1,\ldots,\bsX_{z_n\dhinterseq+\dhinterseq-1}$ used in the definition of $\Psi_1^{(l)}(H)$ are separated
by $\dhinterseq+2$ from the random variables that define $\Psi_2^{(l)}(H)$. The mixing condition \eqref{eq:mixing-rates-0} allows us to replace $\bsX$ with the independent blocks process, that is, we can treat the random variables $\Psi_j^{(l)}(H)$, $j=1,\ldots,\widetilde{m}_n$, as independent. The same applies to $\Psi_j^{(s)}(H)$.

Set
\begin{align}\label{eq:Zn-process}
{\mathbb{Z}}_n(H)=\sum_{j=1}^{\widetilde{m}_n}\left\{Z_{n,j}(H)-\esp[Z_{n,j}(H)]\right\}=:
\sum_{j=1}^{\widetilde{m}_n}\bar Z_{n,j}(H)
\end{align}
with
\begin{align}\label{eq:Zn-process-summands}
Z_{n,j}(H)=\frac{\sqrt{\statinterseqn}}{q_n \dhinterseq \pr(\norm{\bsX_0}>\tepseq)}\Psi_j^{(l)}(H)\;.
\end{align}
The next steps are standard.
\begin{itemize}
\item
First, we show that the limiting variance of the large blocks process ${\mathbb{Z}_n}$ is the same as that of the process $\mathbb{F}_n$;
\item Next, we show that the small blocks process (the second term in
\eqref{eq:large-small-block-process}) is negligible;
\item Finally, we will verify the Lindeberg condition for the large blocks process.
\end{itemize}

\noindent {\bf Variance of the large blocks.}
We have (using the assumed independence of $\Psi_j^{(l)}(H)$)
\begin{align}\label{eq:variance-large-blocks}
&\statinterseqn\var\left(\frac{1}{q_n\dhinterseq \pr(\norm{\bsX_0}>\tepseq)}\sum_{j=1}^{\widetilde{m}_n}\Psi_j^{(l)}(H)\right)
=\frac{\statinterseqn \widetilde{m}_n}{(q_n\dhinterseq \pr(\norm{\bsX_0}>\tepseq))^2}\var(\Psi_1^{(l)}(H))\notag\\
&\sim\frac{1}{ z_n\dhinterseq^3 \pr(\norm{\bsX_0}>\tepseq)}\var\left(\sum_{i=0}^{z_n\dhinterseq-1}H\left(\bsX_{i+1,i+\dhinterseq}/\tepseq\right)\right)
\notag\\
&= \frac{1}{ z_n\dhinterseq^3 \pr(\norm{\bsX_0}>\tepseq)}\var\left(\sum_{j=1}^{z_n}\Psi_j(H)\right)
\;,
\end{align}
where in the last line we decomposed the block $L_1=\{0,\ldots,z_n\dhinterseq-1\}$ into $z_n$ disjoint blocks $J_1,\ldots,J_{z_n}$, used the notation \eqref{eq:block-of-size-rn}, the asymptotics \eqref{eq:ratio} and $\widetilde{m}_n\sim m_n/z_n$.

The next steps are a repetition of the proof of \Cref{prop:covariance-Gn}, with the appropriate adjustments. The term in \eqref{eq:variance-large-blocks} becomes
\begin{align*}
\frac{\var\left(\Psi_1(H)\right)}{\dhinterseq^3\pr(\norm{\bsX_0}>\tepseq)}
+2\frac{1}{\dhinterseq^3 \pr(\norm{\bsX_0}>\tepseq)}\sum_{j=1}^{z_n-1}\left(1-\frac{j}{z_n}\right)\cov(\Psi_1(H),\Psi_{1+j}(H))
\end{align*}
and as in \eqref{eq:ourdecomp} we can write it as
\begin{align}\label{eq:variance-of-the-process}
\frac{1}{\dhinterseq^3\pr(\norm{\bsX_0}>\tepseq)}\left\{\cov(\Psi_2(H),\Psi_1(H)+\Psi_2(H)+\Psi_3(H))\right\}+
\widetilde{A}_n(H)
\end{align}
with the reminder $\widetilde{A}_n(H)$ given this time by (cf. \eqref{eq:ourdecomp1}-\eqref{eq:ourdecomp2})
\begin{align}
&\widetilde{A}_n(H):=-2\frac{1}{z_n}\frac{1}{\dhinterseq^3 \pr(\norm{\bsX_0}>\tepseq)}
\cov(\Psi_1(H),\Psi_2(H))\notag \\
&+\frac{2}{\dhinterseq^3 \pr(\norm{\bsX_0}>\tepseq)} \sum_{j=2}^{z_{n}-1}\left(1-\frac{j}{z_{n}}\right)\left\{\cov\left(\Psi_1(H), \Psi_{1+j}(H)\right)\right\}=:\widetilde{A}_{n,1}(H)+\widetilde{B}_n(H)\;.
\label{eq:reminder-big-blocks}
\end{align}
The reminder is negligible by the same argument as before. Indeed, we note that
$\widetilde{B}_n(H)$ is just $B_n(H)$ from \eqref{eq:ourdecomp2} with $m_n$ replaced with $z_n$. The dependence on $m_n$ vanishes in the final stage of the proof of
\Cref{lem:term-Bn}.
The leading term in \eqref{eq:variance-of-the-process} is the same as in the proof
of \Cref{prop:covariance-Gn}; cf. \eqref{eq:ourdecomp}.

In summary, the variance of the large block process is
\begin{align*}
\lim_{n\to\infty}\var\left({\mathbb{Z}}_n(H)\right)=\tailmeasurestar(H^2)\;.
\end{align*}

\noindent {\bf Variance of the small blocks.} We have (using again the assumed independence of $\Psi_j^{(s)}(H)$ thanks to the beta-mixing)
\begin{align*}
\statinterseqn\var\left(\frac{1}{q_n\dhinterseq \pr(\norm{\bsX_0}>\tepseq)}\sum_{j=1}^{\widetilde{m}_n}\Psi_j^{(s)}(H)\right)
\sim \frac{1}{ z_n\dhinterseq^3 \pr(\norm{\bsX_0}>\tepseq)}\var(\Psi_1^{(s)}(H))\;.
\end{align*}
Since $\Psi_1^{(s)}(H)$ is just $\Psi_1(H)$ defined in \eqref{eq:block-of-size-rn}, we have
\begin{align*}
\statinterseqn\var\left(\frac{1}{q_n\dhinterseq \pr(\norm{\bsX_0}>\tepseq)}\sum_{j=1}^{\widetilde{m}_n}\Psi_j^{(s)}(H)\right)
=O(1/z_n)=o(1)\;.
\end{align*}
\noindent {\bf Lindeberg condition for ${\mathbb Z}_n(H)$.}
We need to show that for all $\eta>0$,
\begin{align}\label{eq:lindeberg}
\lim_{n\to\infty}\widetilde{m}_n\esp\left[Z_{n,1}^2(H)\ind{|Z_{n,1}(H)|> \eta}\right]=0\;.
\end{align}
Since $H$ is bounded, then by \eqref{eq:zn},
\begin{align*}
|Z_{n,1}(H)|\leq \frac{\sqrt{\statinterseqn}z_n\dhinterseq}{q_n \dhinterseq \pr(\norm{\bsX_0}>\tepseq)}\|H\|_{\infty}\sim
\frac{z_n}{\sqrt{n\pr(\norm{\bsX_0}>\tepseq)}}\|H\|_{\infty}=o(1)\;.
\end{align*}
Thus, the indicator in \eqref{eq:lindeberg} becomes zero for large $n$.

\noindent {\bf Lindeberg condition for ${\mathbb Z}_n(\exc)$.}
The functional $\exc$ is not bounded and we will prove the Lindeberg condition under \ref{eq:conditionS}. Write
\begin{align*}
\widetilde{w}_n=\frac{\sqrt{\statinterseqn}}{q_n \dhinterseq \pr(\norm{\bsX_0}>\tepseq)}
\end{align*}
so that
\begin{align*}
&Z_{n,1}(\exc)=\widetilde{w}_n \sum_{i=0}^{z_n\dhinterseq-1}\exc\left(\bsX_{i+1,i+\dhinterseq}/\tepseq\right)
=\widetilde{w}_n \sum_{i=0}^{z_n\dhinterseq-1}\sum_{j=i+1}^{i+\dhinterseq}\ind{\norm{\bsX_j}>\tepseq}\\
&=\widetilde{w}_n \left\{\sum_{j=1}^{\dhinterseq}\sum_{i=0}^{j-1}+\sum_{j=\dhinterseq+1}^{\dhinterseq(z_n+1)}
\sum_{i=j-\dhinterseq}^{j-1}\right\}\ind{\norm{\bsX_j}>\tepseq}\leq \widetilde{w}_n\dhinterseq
\sum_{j=1}^{\dhinterseq(z_n+1)}\ind{\norm{\bsX_j}>\tepseq}\\
&\leq
\frac{\sqrt{\statinterseqn}}{q_n \dhinterseq \pr(\norm{\bsX_0}>\tepseq)}\dhinterseq
\sum_{j=1}^{2\dhinterseq}\ind{\norm{\bsX_j}>\tepseq}
=\frac{1+o(1)}{\sqrt{n\pr(\norm{\bsX_0}>\tepseq)}}\sum_{j=1}^{2\dhinterseq}\ind{\norm{\bsX_j}>\tepseq}\;.
\end{align*}
The last term can be recognized as one (scaled) block of size $2\dhinterseq$ of the tail empirical process
$\TEPunivsl_n(s)$. \cite[Lemma 3.6]{kulik:soulier:wintenberger:2018} (see also \cite[Lemma 9.2.8]{kulik:soulier:2020}) gives
\begin{align*}
    \lim_{n\to\infty}m_n\esp\left[Z_{n,1}^2(\exc)\ind{|Z_{n,1}(\exc)|> \eta}\right]=0 \; .
  \end{align*}
 If moreover \ref{eq:rnbarFun0} holds then
  \begin{align*}
    \lim_{n\to\infty}m_n\esp\left[\bar Z_{n,1}^2(\exc)\ind{|\bar Z_{n,1}(\exc)|> \eta}\right]=0 \; .
  \end{align*}
  Since $\widetilde{m}_n=o(m_n)$, we obtain the Lindeberg condition for ${\mathbb{Z}_n}(\exc)$.
\subsection{Proof of \Cref{thm:sliding-blocks-process-clt} - asymptotic equicontinuity}\label{sec:tightness}
We need the following lemma which is an adapted version of Theorem~2.11.1 in \cite{vandervaart:wellner:1996}.
Let $\mathbb{Z}_n$ be the empirical process indexed by a \semimetric\ space $(\mcg,\metricmcg)$, defined by
\begin{align*}
  \bbZ_n(f) = \sum_{j=1}^{\widetilde{m}_n}\left\{Z_{n,j}(f)-\esp[Z_{n,j}(f)]\right\} \; ,
\end{align*}
where $\{Z_{n,j},n\geq 1\}$, $j=1,\ldots,\widetilde{m}_n$, are \iid\ separable, stochastic processes and $\widetilde{m}_n$
is a sequence of integers such that $\widetilde{m}_n\to\infty$.
Define the random \semimetric\ $d_n$ on $\mcg$ by
\begin{align*}
  d_n^2(f,g) = \sum_{j=1}^{\widetilde{m}_n} \{Z_{n,j}(f)-Z_{n,j}(g)\}^2 \; , f,g\in \mcg \; .
\end{align*}
\begin{lemma}
  \label{theo:VW2.11.1}
Assume that $(\mcg,\metricmcg)$ is totally bounded.  Assume moreover that:
  \begin{enumerate}[(i)]
  \item \label{item:lindeberg-envelope}  For all $\eta>0$,
    \begin{align*}
      \lim_{n\to\infty} {\widetilde{m}_n} \esp[\|Z_{n,1}\|^2_\mcg\ind{\|Z_{n,1}\|^2_\mcg>\eta}] = 0 \; .
    \end{align*}
  \item \label{item:continuity-l2} For every sequence $\{\delta_n\}$ which decreases to zero,
    \begin{align}\label{eq:continuity-l2}
      \lim_{n\to\infty} \sup_{f,g\in\mcg\atop \rho(f,g)\leq\delta_n} \esp[d_n^2(f,g)] = 0 \; . 
    \end{align}
  \item \label{item:random-entropy} There exists a measurable majorant $N^*(\mcg,d_n,\epsilon)$
    of the covering number $N(\mcg,d_n,\epsilon)$ such that for every sequence $\{\delta_n\}$ which
    decreases to zero,
    \begin{align}
      \int_0^{\delta_n} \sqrt{\log N^*(\mcg,d_n,\epsilon)} \rmd\epsilon\convprob 0  \; .
      \label{eq:randomentropy}
    \end{align}
  \end{enumerate}
  Then $\{\mathbb{Z}_n,n\geq 1\}$ is asymptotically $\rho$-equicontinuous, \ie\ for each $\eta>0$,
    \begin{align*}
      \lim_{\delta\to0} \limsup_{n\to\infty} \pr\left(\sup_{{f,g\in \mcg}\atop{\rho(f,g)<\delta}} |\mathbb{Z}_n(f)-\mathbb{Z}_n(g)| > \eta\right) = 0 \; .
    \end{align*}
\end{lemma}
\begin{remark}
   The
  separability assumption is not in \cite{vandervaart:wellner:1996}.  It implies measurability of
  $\|Z_{n,1}\|_{\mcg}$.  Furthermore, the separability also implies that for all $\delta>0$,
  $n\in\Nset$, $(e_j)_{1\leq j\leq \widetilde{m}_n}\in\{-1,0,1\}^{\widetilde{m}_n}$ and $i\in \{1,2\}$, the supremum
    \begin{align*}
      \sup_{f,g\in\mcg \atop \metricmcg(f,g) <\delta} \left|\sum_{j=1}^{\widetilde{m}_n} e_j
      \left(Z_{n,j}(f)-Z_{n,j}(g)\right)^i\right|
      & = \sup_{f,g\in\mcg_0 \atop \metricmcg(f,g) <\delta} \left|\sum_{j=1}^{\widetilde{m}_n} e_j \left(Z_{n,j}(f)-Z_{n,j}(g)\right)^i\right| \;
    \end{align*}
  is measurable, which is an assumption of  \cite{vandervaart:wellner:1996}.
  \remarkend
\end{remark}
\subsubsection{Asymptotic equicontinuity of the empirical process of sliding blocks}\label{sec:tightness-1}
Recall the big-blocks process ${\mathbb{Z}}_n(H)$ (cf. \eqref{eq:Zn-process}-\eqref{eq:Zn-process-summands}).
Recall also that thanks to $\beta$-mixing we can consider random variables $\Psi_j^{(l)}(H)$, $j=1,\ldots,\widetilde{m}_n$ to be independent.
We need to prove asymptotic equicontinuity of ${\mathbb{Z}}_n(H)$
indexed by the class $\mcg = \{H_s,s\in[s_0,t_0]\}$ equipped with the metric $\ltwotmsmetric(H,\widetilde{H})=\tailmeasurestar(\{H-\widetilde{H}\}^2)$.
The same argument can be used to prove asymptotic equicontinuity for the small blocks process. This yields asymptotic equicontinuity of
${\mathbb{F}}_n(H_{\cdot})$. We note further that asymptotic continuity of ${\mathbb{F}}_n(\exc_{\cdot})$ follows from
\cite{kulik:soulier:wintenberger:2018}.
\begin{enumerate}[$\bullet$,wide=0pt]
\item The Lindeberg condition~\ref{item:lindeberg-envelope} of~\Cref{theo:VW2.11.1}   holds because the class $\mcg$ is
  linearly ordered and by applying \eqref{eq:lindeberg}.
\item Since $\mcg$ is linearly ordered, the random entropy condition
  (\ref{eq:randomentropy}) of \Cref{theo:VW2.11.1} holds.

\item Define the random metric
  \begin{align*}
    d_n^2(H,\widetilde{H}) =
    \sum_{j=1}^{\widetilde m_n} (Z_{n,j}(H)-Z_{n,j}(\widetilde{H}))^2 \; .
  \end{align*}
  We need to evaluate
  $\esp[d_n^2(H_s,H_t)]$:
  \begin{align}\label{eq:tightness-evaluation-1}
&\esp[d_n^2(H_s,H_t)]\notag\\
&
=\frac{\statinterseqn\widetilde{m}_n}{(q_n \dhinterseq \pr(\norm{\bsX_0}>\tepseq))^2}
\esp\left[\left(\sum_{i=0}^{z_n\dhinterseq-1}\left\{H_s\left(\bsX_{i+1,i+\dhinterseq}/\tepseq\right)-
H_t\left(\bsX_{i+1,i+\dhinterseq}/\tepseq\right)\right\}\right)^2\right]
\notag\\
&\sim \frac{1}{ z_n\dhinterseq^3 \pr(\norm{\bsX_0}>\tepseq)}
\esp\left[\left(\sum_{j=1}^{z_n}\left\{\Psi_j(H_s)-\Psi_j(H_t)\right\}\right)^2\right]
\;,
\end{align}
where in the last line we decomposed the block $L_1$ into $z_n$ disjoint blocks $J_1,\ldots,J_{z_n}$, used the notation \eqref{eq:block-of-size-rn}, the asymptotics \eqref{eq:ratio} and $\widetilde{m}_n\sim m_n/z_n$; cf. \eqref{eq:variance-large-blocks}.

The term in \eqref{eq:tightness-evaluation-1} becomes
\begin{align*}
&\frac{\esp[(\Psi_1(H_s)-\Psi_1(H_t))^2]}{\dhinterseq^3\pr(\norm{\bsX_0}>\tepseq)}\\
&+2\frac{1}{\dhinterseq^3 \pr(\norm{\bsX_0}>\tepseq)}\sum_{j=1}^{z_n-1}\left(1-\frac{j}{z_n}\right)
\esp[\left\{\Psi_1(H_s)-\Psi_1(H_t)\right\}\left\{\Psi_{1+j}(H_s)-\Psi_{1+j}(H_t)\right\}]
\end{align*}
and as in \eqref{eq:ourdecomp} we can write it as
\begin{align}\label{eq:tightness}
\frac{1}{\dhinterseq^3\pr(\norm{\bsX_0}>\tepseq)}
\esp\left[\left\{\Psi_2(H_s)-\Psi_2(H_t)\right\}\sum_{j=1}^2\left\{\Psi_j(H_s)-\Psi_j(H_t)\right\}\right\}+
\widetilde{A}_n(H,s,t)
\end{align}
with the reminder (cf. \eqref{eq:reminder-big-blocks})
\begin{align*}
&\widetilde{A}_n(H,s,t):=-2\frac{1}{z_n}\frac{1}{\dhinterseq^3 \pr(\norm{\bsX_0}>\tepseq)}
\esp[\left\{\Psi_1(H_s)-\Psi_1(H_t)\right\}\left\{\Psi_2(H_s)-\Psi_2(H_t)\right\}] \\
&+2 \frac{1}{\dhinterseq^3 \pr(\norm{\bsX_0}>\tepseq)} \sum_{j=2}^{z_{n}-1}\left(1-\frac{j}{z_{n}}\right)
\esp\left[\left\{\Psi_1(H_s)-\Psi_1(H_t)\right\} \left\{\Psi_{j+1}(H_s)-\Psi_{j+1}(H_t)\right\}\right]\\
&= \widetilde{A}_{n,1}(H_s-H_t)+\widetilde{B}_n(H_s-H_t)\;.
\end{align*}
\Cref{rem:A-LinA} applies and hence by \Cref{lem:term-Bn},
\begin{align*}
\lim_{n\to\infty}\sup_{s\in [s_0,t_0]}\widetilde{B}_n(H_s-H_t)=0\;.
\end{align*}
The leading term in \eqref{eq:tightness} is decomposed as (cf. \eqref{eq:variance-decomposition-leading-term})
\begin{align*}
\frac{1}{\dhinterseq} \widetilde{g}_{n}(0;H_s-H_t)+
2\frac{1}{\dhinterseq}
\sum_{i=1}^{\dhinterseq-1}\widetilde{g}_{n}(i/\dhinterseq;H_s-H_t)+
2R_n(H_s-H_t)\;.
\end{align*}
Again, \Cref{rem:A-LinA} applies and \eqref{eq:reminder-2-mcg} gives
\begin{align*}
\lim_{n\to\infty}\sup_{s\in [s_0,t_0]}R_n(H_s-H_t)=0\;.
\end{align*}
It remains to show that for every
  sequence $\{\delta_n\}$ decreasing to zero,
  \begin{align*}
    \lim_{n\to\infty} \sup_{s,t\in[s_0,t_0] \atop |s-t|\leq \delta_n}
    \frac{1}{\dhinterseq}
\sum_{i=1}^{\dhinterseq-1}\widetilde{g}_{n}(i/\dhinterseq;H_s-H_t)=
 \lim_{n\to\infty} \sup_{s,t\in[s_0,t_0] \atop |s-t|\leq \delta_n}\int_0^1\widetilde{g}_n(\xi,H_s-H_t)\rmd \xi
=
0\;.
  \end{align*}
Because of the monotonicity
\begin{align*}
|\widetilde{g}_n(\xi,H_s-H_t)|\leq 2\sup_{s\in [s_0,t_0]}|H_s| \tailmeasurestar_{n,\dhinterseq}(|H_s-H_t|)
\leq 2\max\{|H_{s_0}|,|H_{t_0}|\}|\tailmeasurestar_{n,\dhinterseq}(H_s)-\tailmeasurestar_{n,\dhinterseq}(H_t)|\;.
\end{align*}
The
  convergence of $\tailmeasurestar_{n,\dhinterseq}(H_s)$ to
  $s^{-\alpha}\tailmeasurestar(H^2)$ is uniform on $[s_0,t_0]$. Thus,
  for $s,t\in[s_0,t_0]$,
  \begin{multline*}
    \left| \tailmeasurestar_{n,\dhinterseq}(H_s)  - \tailmeasurestar_{n,\dhinterseq}(H_{t}) \right|
    \leq 2 \sup_{s_0\leq u \leq t_0} \left| \tailmeasurestar_{n,\dhinterseq}(H_u) -
      \tailmeasurestar(H_{u}) \right| + \tailmeasurestar(H) \{s^{-\alpha} - t^{-\alpha}\} \; .
  \end{multline*}
  Fix $\eta>0$. For large enough $n$, the uniform convergence yields
  \begin{multline*}
    \sup_{s_0\leq s, t \leq t_0 \atop |s-t|\leq\delta_n} \left|
      \tailmeasurestar_{n,\dhinterseq}(H_s) - \tailmeasurestar_{n,\dhinterseq}(H_{t}) \right|
    \leq \eta + \tailmeasurestar(H) \sup_{s_0\leq s, t \leq t_0 \atop |s-t|\leq\delta_n}  \{s^{-\alpha} - t^{-\alpha}\} \\
    \leq \eta + \alpha s_0^{-\alpha-1} \delta_n \tailmeasurestar(H)\; .
  \end{multline*}
  This proves that \eqref{eq:continuity-l2} holds.
\end{enumerate}
The conditions of \Cref{theo:VW2.11.1} hold, thus the sequence ${\mathbb{Z}_n}$ is asymptotically
equicontinuous.

\subsection{Proof of \Cref{thm:sliding-block-clt-1}}\label{sec:proof-conclusion}
Write $\zeta_{n}=\orderstat[\norm{\bsX}]{n}{n-\statinterseqn}/\tepseq$. Since
  $\statinterseqn=n\pr(\norm{\bsX_0}>\tepseq)$, we have the relationship $\TEDclusterrandomsl_{n,\dhinterseq}(H)=\tedclustersl(H_{\zeta_n})$ (cf. \eqref{eq:sliding-block-estimator-nonfeasible-1}-\eqref{eq:sliding-block-estimator-feasible}). Therefore,
  \begin{align}
    \label{eq:tep-cluster-decomp}
    &\sqrt{\statinterseqn}\left\{\TEDclusterrandomsl_{n,\dhinterseq}(H)-\tailmeasurestar(H)\right\}
      = \mathbb{F}_n(H_{\zeta_n})   + \sqrt{\statinterseqn} \left\{\tailmeasurestar(H_{\zeta_n}) - \tailmeasurestar(H)\right\}    \; .
  \end{align}
  \begin{enumerate}[Step 1.,wide=0pt]
  \item \Cref{thm:sliding-blocks-process-clt} gives local uniform convergence of $\{\mathbb{F}_n(H_{s}), s\in [s_0,t_0]\}$ to a continuous Gaussian process $\mathbb{G}$. At the same time, convergence of
      $\{\mathbb{F}_n(\exc_{s}), s\in [s_0,t_0]\}$ yields $\zeta_n\convprob1$, jointly with $\mathbb{F}_n(H_{s})$. Therefore, $\mathbb{F}_n(H_{\zeta_n})\convdistr \mathbb{G}(H)$.
  \item Using Vervaat's theorem, we have, jointly with the previous
    convergence,
      $\sqrt{k} (\zeta_n^{-\alpha}-1)  \convdistr \tcb{-}\  \TEPclusterlimit(\exc)$.
    Therefore, by the homogeneity of $\tailmeasurestar$,
    \begin{align*}
      \sqrt{k} \left\{\tailmeasurestar(H_{\zeta_n}) - \tailmeasurestar(H)\right\}
      =   \tailmeasurestar(H)      \sqrt{k} (\zeta_n^{-\alpha}-1)  \convdistr - \tailmeasurestar(H) \TEPclusterlimit(\exc) \; .
    \end{align*}
  \end{enumerate}
  Since the convergences hold jointly, we conclude the result.

\subsection{Auxiliary results}
\begin{lemma}[Problems 5.24 and 5.25 in \cite{kulik:soulier:2020}]
  \label{exer:expression-tailmeasurestar-HH'}
  Assume that $\pr(\lim_{|j|\to\infty} \norm{\bsY_j}=0)=1$ and let $H$, $H'$ be bounded functionals
   on $(\Rset^d)^\Zset$ such that $H'(\bsx)=0$ if $\bsx^*\leq 1$  and
  $\esp[|H(\bsY)| |H'(\bsY_{0,\infty})-H'(\bsY_{1,\infty})|] <\infty$. Then
  \begin{align*}
     &\tailmeasurestar(HH') = \esp[H(\bsY) \{H'(\bsY_{0,\infty})-H'(\bsY_{1,\infty})\}] \;,\\
     &\tailmeasurestar(H\exc)=\esp[H(\bsY)]\;, \ \
    \tailmeasurestar(\exc)=1\;, \ \
    \tailmeasurestar(\exc^2) = \sum_{j\in\Zset} \pr(\norm{\bsY_{j}}>1)\;.
   \end{align*}
  \end{lemma}
   \begin{proof}
     Applying \eqref{eq:relation-cluster-Y-Q-Theta-epsilon} and the time change formula (see \cite[Lemma 2.2]{planinic:soulier:2018}), we obtain
     \begin{align*}
       \tailmeasurestar(HH') &= \esp[H(\bsY)H'(\bsY) \ind{\bsY_{-\infty,-1}^*\leq1}] = \esp[|H(\bsY)| |H'(\bsY_{0,\infty})| \ind{\bsY_{-\infty,-1}^*\leq1}]    \\
       & \leq  \sum_{j=0}^\infty  \esp[|H(\bsY)| |H'(\bsY_{j,\infty})-H'(\bsY_{j+1,\infty})|
         \ind{\norm{\bsY_j}>1} \ind{\bsY_{-\infty,-1}^*\leq1}]    \\
       & = \sum_{j=0}^\infty  \esp[H(\bsY) |H'(\bsY_{0,\infty})-H'(\bsY_{1,\infty})| \ind{\norm{\bsY_{-j}}>1} \ind{\bsY_{-\infty,-j-1}^*\leq1}]    \\
       & = \esp[H(\bsY) |H'(\bsY_{0,\infty})-H'(\bsY_{1,\infty})|] <\infty \; .
     \end{align*}
     This proves that $\tailmeasurestar(HH')<\infty$. Hence, we can switch the expectation with the summation and the first result follows. The second statement follows by noting that $\exc(\bsY_{0,\infty}) - \exc(\bsY_{1,\infty}) = 1$ almost surely.
   \end{proof}
\begin{lemma}[Example 6.2.2 and Problem 6.7 in \cite{kulik:soulier:2020}]
\label{lem:cluster-tail}
Assume that $\pr(\lim_{|j|\to\infty} \norm{\bsY_j}=0)=1$ and let $\pi(m)$, $m\geq 0$, be the limiting cluster size distribution. Then
\begin{align*}
\sum_{m=1}^\infty m\pi(m) &= \canditheta^{-1}\;, \ \ \sum_{m=1}^\infty m^2\pi(m)  = \canditheta^{-1} \sum_{j\in\Zset}\pr(\norm{\bsY_j}>1)\;.
\end{align*}
\end{lemma}
\begin{proof}
For the first statement, applying
  \eqref{eq:relation-cluster-Y-Q-Theta-epsilon} and \Cref{exer:expression-tailmeasurestar-HH'}, we have,
  \begin{align*}
    \sum_{m=1}^\infty m\pi(m)
    & = \sum_{m=1}^\infty m \pr(\exc(\bsY) = m\mid \anchor_0(\bsY)=0)
      = \esp[\exc(\bsY) \mid \anchor_0(\bsY)=0]  = \canditheta^{-1} \tailmeasurestar(\exc)  = \canditheta^{-1} \; .
  \end{align*}
Likewise,
\begin{align*}
  \sum_{m=1}^\infty m^2\pi(m)
  &  = \esp[\exc^2(\bsY) \mid \anchor_0(\bsY)=0]  = \canditheta^{-1} \tailmeasurestar(\exc^2) =
    \canditheta^{-1} \sum_{j\in\Zset}\pr(\norm{\bsY_j}>1)\;.
  \end{align*}
\end{proof}

\end{document}